\documentclass[11pt,a4paper,reqno]{amsart}

\usepackage{a4wide}
\usepackage{amssymb,amsmath,amsthm,mathrsfs}
\usepackage{graphicx}
\usepackage{float}
\usepackage{colortbl}
\usepackage{enumerate}
\usepackage{upref}
\usepackage{stackrel}
\usepackage[bookmarks=false]{hyperref}
\usepackage[T2A,T1]{fontenc}
\DeclareSymbolFont{cyrillic}{T2A}{cmr}{m}{n}
\DeclareMathSymbol{\D}{\mathalpha}{cyrillic}{196}

\theoremstyle{plain}% default
\newtheorem{theorem}{Theorem}[section]
\newtheorem{lemma}[theorem]{Lemma}

\newtheorem{propositionP}[theorem]{Proposition}
\newtheorem{proposition}[theorem]{Proposition}
\newtheorem{corollary}[theorem]{Corollary}

\theoremstyle{definition}
\newtheorem*{condition}{Condition}

\theoremstyle{remark}
\newtheorem{remark}[theorem]{Remark}%[section]

\usepackage{enumitem}
\makeatletter
\def\namedlabel#1#2{\begingroup
   #2%
 \def\@currentlabel{#2}%
   \phantomsection\label{#1}\endgroup
}

\def\XI{\mathbf{\Xi}}

\def\o{\omega}
\def\R{\ensuremath{\mathbb R}}
\def\N{\ensuremath{\mathbb N}}

\def\I{\ensuremath{{\bf 1}}}
\def\e{{\ensuremath{\rm e}}}

\def\esup{\ensuremath{\rm essup}}

\def\L{\ensuremath{\mathcal L}}

\def\p{\ensuremath{\mathbb P}}
\def\F{\ensuremath{\mathcal F}}

\def\o{\ensuremath{\omega}}

\def\A{\ensuremath{A^{(q)}}}

\def\Y{\mathcal{Y}}

\def\1{{\bf 1}}

\def\ie{{\em i.e.}, }

\def\dist{\ensuremath{\text{dist}}}

\def\TF{\mathcal{T}}

\def\eps{\varepsilon}

\numberwithin{equation}{section}

\begin{document}

\title[EVL for non stationary processes generated by sequential dynamics]{Extreme Value Laws for non stationary processes generated by sequential and random dynamical systems}

\author[A. C. M. Freitas]{Ana Cristina Moreira Freitas}
\address{Ana Cristina Moreira Freitas\\ Centro de Matem\'{a}tica \&
Faculdade de Economia da Universidade do Porto\\ Rua Dr. Roberto Frias \\
4200-464 Porto\\ Portugal} \email{\href{mailto:amoreira@fep.up.pt}{amoreira@fep.up.pt}}
\urladdr{\url{http://www.fep.up.pt/docentes/amoreira/}}

\author[J. M. Freitas]{Jorge Milhazes Freitas}
\address{Jorge Milhazes Freitas\\ Centro de Matem\'{a}tica \& Faculdade de Ci\^encias da Universidade do Porto\\ Rua do
Campo Alegre 687\\ 4169-007 Porto\\ Portugal}
\email{\href{mailto:jmfreita@fc.up.pt}{jmfreita@fc.up.pt}}
\urladdr{\url{http://www.fc.up.pt/pessoas/jmfreita/}}

\author[S. Vaienti]{Sandro Vaienti}
\address{Sandro Vaienti\\ Aix Marseille Universit\'e, CNRS, CPT, UMR 7332
\\ 13288 Marseille, France and
Universit\'e de Toulon, CNRS, CPT, UMR 7332, 83957 La Garde, France.}
\email{vaienti@cpt.univ-mrs.fr}
\urladdr{\url{http://www.cpt.univ-mrs.fr/~vaienti/}}
%
%

%\thanks{
%ACMF was partially supported by FCT grant SFRH/BPD/66174/2009. JMF was partially supported by FCT grant SFRH/BPD/66040/2009. All these grants are financially supported by the program POPH/FSE.  ACMF and JMF were partially supported by CMUP (UID/MAT/00144/2013), which is funded by FCT (Portugal) with national (MEC) and European structural funds through the programs FEDER, under the partnership agreement PT2020. SV was supported by the ANR- Project {\em Perturbations} and by the project {\em Atracci\'on de Capital Humano Avanzado del Extranjero} MEC 80130047, CONICYT, at the CIMFAV, University of Valparaiso. All authors were partially  supported by FCT project PTDC/MAT/120346/2010, which is funded by national and European structural funds through the programs  FEDER and COMPETE. SV is grateful to  N. Haydn, M. Nicol and A. T\"or\"ok for several and fruitful discussions on sequential systems, especially in the framework of indifferent maps. JMF is grateful to M. Todd for fruitful discussions and careful reading of a preliminary version of this paper.  All authors acknowledge the Isaac Newton Institute for Mathematical Sciences, where this work was initiated during the program Mathematics for the Fluid Earth.}

\date{\today}


%37A50  	Relations with probability theory and stochastic processes
%60G70  	Extreme value theory; extremal processes
%37B20  	Notions of recurrence
%60G10  	Stationary processes
%37C25  	Fixed points, periodic points, fixed-point index theory
%--------
%37A25  	Ergodicity, mixing, rates of mixing
%37D25  	Nonuniformly hyperbolic systems (Lyapunov exponents, Pesin theory, etc.)
%37D35  	Thermodynamic formalism, variational principles, equilibrium states
%37C40  	Smooth ergodic theory, invariant measures

\begin{abstract}
We develop and generalize the theory of extreme value for non-stationary stochastic processes, mostly  by weakening the uniform mixing condition that was previously  used in this setting. We apply our results to  non-autonomous dynamical systems, in particular to {\em sequential dynamical systems}, given by uniformly expanding maps,   and to a few classes of {\em random dynamical systems}. Some examples are presented and  worked out in detail.

\vspace{0.5cm}

\noindent \textsc{R\'esum\'e.} \hspace{0.02cm}
Nous d\'eveloppons et g\'en\'eralisons la th\'eorie des valeurs extr\^emes pour des processus stochastiques non-stationnaires, en  affaiblissant   la condition de m\'elange uniforme qui avait \'et\'e utilis\'ee auparavant.  Nous appliquons nos r\'esultats \`a des syst\`emes dynamiques non autonomes, en particulier aux syst\`emes dynamiques s\'equentiels engendr\'es par des applications dilatantes et \`a une large classe de syst\`emes dynamiques al\'eatoires. Quelques exemples sont pr\'esent\'es et calcul\'es  en d\'etail.

\end{abstract}

%\renewcommand{\abstractname}{R\'esum\'e}

%{\bf R\'esum\'e}\\

%%\begin{abstract}
%Nous d\'eveloppons et g\'en\'eralisons la th\'eorie des valeurs extr\^emes pour des processus stochastiques non-stationnaires, en  affaiblissant   la condition de m\'elange uniforme qui avait \'et\'e utilis\'ee auparavant.  Nous appliquons nos r\'esultats \`a des syst\`emes dynamiques non autonomes, en particulier aux syst\`emes dynamiques s\'equentiels engendr\'es par des applications dilatantes et \`a une large classe de syst\`emes dynamiques al\'eatoires. Quelques exemples sont pr\'esent\'es et calcul\'es  en d\'etail.
%\end{abstract}
\maketitle
\tableofcontents

\section{Introduction}

\subsection{The motivation and the dynamical setting}
One of the most successful directions of ergodic theory in the last decades was the application of probabilistic tools to characterise the asymptotic evolution of a given dynamical system. There is now a well established domain known as {\em statistical properties of dynamical systems}, which attempts to prove limit theorems under different degrees of mixing. Mixing is the way to restore asymptotic independence and, in this way, mimic independent and identically distributed (i.i.d.)  sequences of random variables. A common distribution for the time series arising from the dynamical systems is acquired from the existence of an invariant measure for such systems. In some sense, the existence of such a measure is what defines a dynamical system. Relaxing this assumption gives rise to non-autonomous dynamical systems for which the study of limit theorems is just at the beginning. In this paper, we will focus on one of those statistical properties, namely on asymptotic extreme value distribution laws. Our first goal will be to improve and generalise the previous results by H\"usler (see below), which held for non-identically distributed random variables but under a uniform mixing condition, to the mixing situations typical in dynamical systems. Then we will apply our theoretical results to
 two important examples of non-stationary processes arising in dynamical systems.

 The first example is given by {\em sequential dynamical systems}; they were introduced by Berend and Bergelson \cite{BB84}, as a non-stationary system in which a  concatenation of maps is applied to a given point in the underlying space, and the probability is taken as  a conformal measure, which is conformal for all maps considered and allows the use of the transfer operator (Perron-Fr\"obenius) as a useful tool to quantify the loss of memory of any prescribed initial observable. The theory of sequential systems was later developed  in the fundamental paper by Conze and Raugi \cite{CR07}, where a few limit theorems, in particular the Central Limit Theorem, were proved for concatenations of  one-dimensional dynamical systems, each possessing a transfer operator with a quasi-compact structure on a suitable Banach space. For the same systems and others, even in higher dimensions, the Almost Sure Invariance Principle was subsequently shown  \cite{HNTV}; we will refer to the large class of systems investigated in \cite{HNTV} as concrete examples to  which the non-stationary extreme value theory presented in this article applies.

  The second example pertains to {\em random transformations}, which are constructed on a skew-system whose base is an invertible and hyperbolic system which codes a map on the second factor (this second factor could be seen as {\em fibers}, which are all copy of the same set). On these fibers live a family of {\em sample measures}, each of them corresponding to different ways to code the orbit of a given point. These sample measures will be taken as the probability measures that describe the statistical properties along the factor and they do not give rise to stationary processes (although they satisfy an interesting property when they move from one fiber to the other). Averaging along a sample measure means to fix the particular initial fiber which supports it; the dynamics will transport this measure  from one fiber to the other, and this non-stationary process could  be assimilated to a quenched process, where the map changes step by step according to a given realization.   We defer to the books by L. Arnold \cite{A98} and Y. Kifer \cite{K86,K88} for a detailed account of these transformations, in particular for their ergodic properties. Limit theorems, in particular the CLT, were investigated in \cite{K98}.
There are a few attempts to investigate recurrence in the framework of random transformations: see for instance
\cite{AFV15, RSV14, R14, RT15, KR14}.

\subsection{Extreme Value Laws for general non-stationary processes}
\label{sec:EVL-nonstationary}

As mentioned in \cite{FHR11}, the class of non-stationary stochastic processes is rather large and an Extreme Value Theory for such a general class does not exist. In \cite{H83,H86}, H\"usler developed the first approach to the subject. Under convenient conditions, one can recover the usual extremal behaviour seen for i.i.d or stationary sequences under Leadbetter's conditions. Of course the degree of freedom involved is so large that it is not difficult to give examples with pathological behaviour (see \cite[Section~3]{H86} or  \cite[Example~9.4.4]{FHR11}). However, for appropriate  subclasses, such as for stochastic processes of the form $X_i=a_i+b_iY_i$, with trend values $a_i$, scaling values $b_i$ and a stationary (or i.i.d) stochastic process $Y_0, Y_1, \ldots$, one can study them and obtain the expected behaviour (see \cite{N97}).

The existing theory of extreme values for non-stationary sequences (which is still mostly based on H\"usler's results, see \cite{FHR11}) is not applicable in a dynamical setting because it is built over a uniform mixing condition obtained by adjusting to the non-stationary setting, Leadbetter's $D(u_n)$ condition for stationary processes. As was seen in the stationary setting in \cite{C01,FF08}, this type of condition is not appropriate for stochastic processes arising from dynamical systems since it does not follow from usual properties regarding the loss of memory of chaotic systems, which are usually formulated in terms of decay of correlations. See discussion in Section~2 of \cite{FFT15} and Remarks~2.1 and 3.5 of the same paper.

Hence, the first goal of this paper is to develop a more general theory of extreme values for non-stationary stochastic processes, which enables the study of the extremal behaviour of the non-stationary systems discussed in the preceding Section. The major highlights of this generalisation are: the use of a much weaker mixing condition, motivated by an idea of Collet (in \cite{C01}) and further developed in \cite{FF08,FFT12,FFT15}, that we will adapt to the non-stationary setting and denote by a cyrilic D, \ie $\D$, as in \cite{FFT15};  and a much more sophisticated way of dealing with clustering and the appearance of an Extremal Index less than 1, which is based on an idea introduced in \cite{FFT12} and further developed in \cite{FFT15}, which basically says that when dealing with clustering due to the presence of a periodic phenomenon we can replace the role of the occurrence of exceedances (which in the dynamical setting correspond to hits to target ball sets) by that of the occurrence of escapes (which in the dynamical setting can be associated with hits to annuli target sets).

While in \cite{H83,H86}, H\"usler built on the existing theory of extreme values for stationary sequences developed by Leadbetter and others, here we will follow H\"usler's approach but adapt to the non-stationary setting the more refined \cite{FFT15}.

\section{A general result for extreme value laws for non-stationary processes}

In this section will try to keep as much as possible the notations used in \cite{H83,H86,FFT15}.

Let $X_0, X_1, \ldots$ be a stochastic process\index{random/stochastic process}, where each r.v. $X_i:\Y\to\R$ is defined on the measure space $(\Y,\mathcal B,\p)$.\\

We assume that $\Y$  is a sequence space with a natural product structure so that each possible realisation of the stochastic process\index{random/stochastic process} corresponds to a unique element of $\Y$ and there exists a measurable map $\TF:\Y\to\Y$, the time evolution map, which can be seen as the passage of one unit of time, so that
\[
X_{i-1}\circ \TF =X_{i}, \quad \mbox{for all $i\in\N$}.
\]
The $\sigma$-algebra $\mathcal B$ can also be seen as a product $\sigma$-algebra adapted to the $X_i$'s.
For the purpose of this paper, $X_0, X_1,\ldots$ is possibly non-stationary. Stationarity would mean that $\p$ is $\TF$-invariant. Note that $X_i=X_0\circ \TF_i$, for all $i\in\N_0$, where $\TF_i$ denotes the $i$-fold composition of $\TF$, with the convention that $\TF_0$ denotes the identity map on $\Y$. In the applications  below to sequential dynamical systems, we will have that $\TF_i=T_i\circ \ldots\circ T_1$ will be the concatenation of $i$ possibly different transformations $T_1, \ldots, T_i$.

Each random variable $X_i$ has a marginal distribution function (d.f.) denoted by $F_i$, \ie $F_i(x)=\p(X_i\leq x)$. Note that the $F_i$, with $i\in\N_0$, may all be distinct from each other. For a d.f. $F$ we let $\bar F=1-F$. We define $u_{F_i}=\sup\{x:F_i(x)<1\}$ and let $F_i(u_{F_i}-):=\lim_{h\to 0, h>0} F_i(u_{F_i}-h) =1$ for all $i$.

Our main goal is to determine the limiting law of
$$\mathbf{P}_n=\p(X_0\leq u_{n,0},X_1\leq u_{n,1},\ldots,X_{n-1}\leq u_{n,n-1})$$
as $n\to\infty$, where $\{u_{n,i},i\leq n-1,n\geq 1\}$ is considered a real-valued boundary.
We assume throughout the paper that
\begin{equation}
\bar F_{n,\max}:=\max\{\bar F_i(u_{n,i}), i\leq n-1\}\to0 \mbox{ as }n\to\infty,
\label{Fmax}
\end{equation}
which is equivalent to
$$u_{n,i}\to u_{F_i}\mbox{ as }n\to\infty,\mbox{ uniformly in } i.$$

Let us denote $F^*_n:=\sum_{i=0}^{n-1}\bar F_i(u_{n,i}),$
and assume that  there is $\tau>0$ such that
\begin{equation}
F^*_n:=\sum_{i=0}^{n-1}\bar F_i(u_{n,i})\to\tau, \qquad \mbox{as $n\to\infty.$}
\label{F_n}
\end{equation}

To simplify the notation let $u_i:=u_{n,i}$.

In what follows, for every $A\in\mathcal B$, we denote the complement of $A$ as $A^c:=\mathcal Y\setminus A$.

Let $\mathbb A:=(A_0,A_1,\ldots)$ be a sequence of events such that $A_i\in \TF_i^{-1} \mathcal B$. For some $s,\ell\in \N_0$, we define
\begin{equation}
\label{eq:W-def}
\mathscr W_{s,\ell}(\mathbb A)=\bigcap_{i=s}^{s+\ell-1} A_i^c.
\end{equation}

We will write $\mathscr W_{s,\ell}^c(\mathbb A):=(\mathscr W_{s,\ell}(\mathbb A))^c$.

For some $j\in \N_0$,  we consider
$$\mathbb A_n^{(j)}:=(A_{n,0}^{(j)},A_{n,1}^{(j)},\ldots),$$ where the event $A_{n,i}^{(j)}$ is defined for $j\in\N$ as
$$A_{n,i}^{(j)}:=\{X_i>u_{n,i},X_{i+1}\leq u_{n,i+1},\ldots,X_{i+j}\leq u_{n,i+j}\}$$
and, for $j=0$, we simply define $A_{n,i}^{(0)}(u_{n,i}):=\{X_i>u_{n,i}\}$.

For each $i\in\N_0$ and $n\in\N$, let $R_{n,i}^{(j)}=\min\{r\in\N: A_{n,i}^{(j)}\cap A_{n,i+r}^{(j)}\neq\emptyset\}$.
We assume that there exists $q\in\N_0$ such that:
\begin{equation}
\label{def:q}
q=\min\left\{j\in\N_0:\lim_{n\to\infty} \min_{i\leq n}\left\{R_{n,i}^{(j)}\right\}=\infty\right\}.
\end{equation}

When $q=0$ then $A_{n,i}^{(0)}(u_{n,i})$ corresponds to an exceedance of the threshold $u_{n,i}$ and we expect no clustering of exceedances.

When $q>0$, heuristically one can think that there exists an underlying periodic phenomenon creating short recurrence, \ie clustering of exceedances,  when exceedances occur separated by no more than $q-1$ units of time then they belong to the same cluster. Hence, the sets $A_{n,i}^{(q)}(u_{n,i})$ correspond to the occurrence of exceedances that escape the periodic phenomenon and are not followed by another exceedance in the same cluster. We will refer to the occurrence of $A_{n,i}^{(q)}(u_{n,i})$ as the occurrence of an escape at time $i$, whenever $q>0$.

The following result adapts to the non-stationary setting an idea introduced in \cite{FFT12} and further developed in \cite[Proposition~2.7]{FFT15}, which essentially says the asymptotic distribution of $\mathbf P_n$ coincides with that of $\mathscr W_{0,n}(\mathbb A_n^{(q)})$, which motivates the special role played by $\mathbb A_n^{(q)}$ and the conditions we propose next.

\begin{proposition}
\label{prop:relation-balls-annuli-general}
Given events $B_0,B_1,\ldots\in\mathcal B$, let $r,q,n\in\N$ be such that $q<n$ and define $\mathbb B=(B_0,B_1,\ldots)$, $A_r=B_r\setminus \bigcup_{j=1}^{q} B_{r+j}$ and $\mathbb A=(A_0,A_1,\ldots)$. Then
$$
\left|\p(\mathscr W_{0,n}(\mathbb B))-\p(\mathscr W_{0,n}(\mathbb A))\right|\leq \sum_{j=1}^{q} \p\left(\mathscr W_{0,n}(\mathbb A)\cap (B_{n-j}\setminus A_{n-j})\right).
$$
\end{proposition}

Now, we introduce a mixing condition which is specially designed for the application to the dynamical setting, contrary to the existing ones in the literature.

\begin{condition}[$\D_q(u_{n,i})$]\label{cond:D} We say that $\D_q(u_n)$ holds for the sequence $X_0,X_1,\ldots$ if for every  $\ell,t,n\in\N$,
\begin{equation}\label{eq:D1}
\left|\p\left(\A_{n,i}\cap
 \mathscr W_{i+t,\ell}\left(\mathbb A_n^{(q)}\right) \right)-\p\left(\A_{n,i}\right)
  \p\left(\mathscr W_{i+t,\ell}\left(\mathbb A_n^{(q)}\right)\right)\right|\leq \gamma_i(q,n,t),
\end{equation}
where $\gamma_i(q,n,t)$ is decreasing in $t$ for each $n$ and each $i$ and there exists a sequence $(t_n^*)_{n\in\N}$ such that $t_n^* \bar F_{n,\max}\to0$ and
$\sum_{i=0}^{n-1}\gamma_i(q,n,t_n^*)\to0$ when $n\rightarrow\infty$.
\end{condition}

\begin{remark}Condition $\D_q(u_{n,i})$ is a sort of mixing condition resembling H\"usler's adjustment of Leadbetter's condition $D(u_n)$ but with the great advantage that it can be checked for non-stationary dynamical systems, as we will see in Sections~\ref{subsec:D-check} %\ref{subsec:D0}
and \ref{subsec:random-subshift},  contrary to H\"usler's $D(u_{n,i})$. This advantage resides on the fact that the event $A_{n,i}^{(q)}(u_{n,i})$ depends only on a finite number of random variables,  making $\D_q(u_{n,i})$ a much weaker requirement in terms of uniformity when compared to H\"usler's $D(u_{n,i})$. Recall that H\"usler's $D(u_{n,i})$ required an uniform bound for all possible $i$ and all possible numbers of random variables of the process on which the first event depended.
\end{remark}

In order to prove the existence of a distributional limit for $\mathbf P_n$ we use as usual a blocking argument that splits the data into $k_n$ blocks separated by time gaps of size larger than $t_n^*$, which are created by simply disregarding the observations in the time frame occupied by the gaps. The precise construction of the blocks is given in Section~\ref{subsec:blocks} but we briefly describe below some of the properties of this construction.

In the stationary context, one takes blocks of equal size, which in particular means that the expected number of exceedances within each block is $n\p(X_0>u_n)/k_n\sim \tau/k_n$. Here the blocks may have different sizes, which we will denote by $\ell_{n,1}, \ldots, \ell_{n, k_n}$ but, as in \cite{H83,H86}, these are chosen so that the expected number of exceedances is again $\sim\tau/k_n$.  Also, for $i=1,\ldots,k_n$, let $\L_{n,i}=\sum_{j=1}^{i}\ell_{n,j}$ and $\L_{n,0}=0$.

The time gaps are created by disregarding the last observations in each block so that the true blocks become the remaining part. To do that, we have to balance the facts that we want the gaps to be big enough so that they are larger than $t_n^*$ but on the other hand we also want  the gaps to be sufficiently  small so that the information disregarded does not compromise the computations. This is achieved by choosing the number of blocks, which correspond to the sequence $(k_n)_{n\in\N}$ diverging but slowly enough so that the weight of the gaps is negligible when  compared to that of the true blocks.

As usual in extreme value theory, in order to guarantee the existence of a distributional limit one needs to impose some restrictions on the speed of recurrence.

For $q\in\N_0$ given by \eqref{def:q}, consider the sequence $(t_n^*)_{n\in\N}$, given by condition  $\D_q(u_n)$ and let $(k_n)_{n\in\N}$ be another sequence of integers such that
\begin{equation}
\label{eq:kn-sequence}
k_n\to\infty\quad \mbox{and}\quad  k_n t_n^* \bar F_{n,\max}\to0
\end{equation}
as $n\rightarrow\infty$.

\begin{condition}[$\D'_q(u_{n,i})$]\label{cond:D'q} We say that $\D'_q(u_{n,i})$
holds for the sequence $X_0,X_1,X_2,\ldots$ if there exists a sequence $(k_n)_{n\in\N}$ satisfying \eqref{eq:kn-sequence} and such that
\begin{equation}
\label{eq:D'rho-un}
\lim_{n\rightarrow\infty}\sum_{i=1}^{k_n} \sum_{j=0}^{\ell_i-1} \sum_{r>j}^{\ell_i-1}\p(\A_{\L_{i-1}+j}\cap \A_{\L_{i-1}+r})=0.
\end{equation}
\end{condition}

Condition $\D'_q(u_{n,i})$ precludes the occurrence of clustering of escapes (or exceedances, when $q=0$).

\begin{remark}
Note that condition $\D'_p(u_{n,i})$ is an adjustment of a similar condition $\D'_p(u_{n})$ in \cite{FFT15} in the stationary setting, which is similar to (although slightly weaker than) condition $D^{(p+1)}(u_n)$ in the formulation of \cite[Equation (1.2)]{CHM91}
\end{remark}

When $q=0$, observe that $\D'_q(u_{n,i})$ is very similar to $D'(u_{n,i})$ from H\"usler, which prevents clustering of exceedances, just as $D'(u_n)$ introduced by Leadbetter did in the stationary setting.

When $q>0$, we have clustering of exceedances, \ie the exceedances have a tendency to appear aggregated in groups (called clusters). One of the main ideas in \cite{FFT12} that we use here is that the events $\A_{n,i}$ play a key role in determining the limiting EVL and in identifying the clusters. In fact, when $\D'_q(u_{n,i})$ holds we have that every cluster ends with an entrance in $\A_{n,i}$, meaning that the inter cluster exceedances must appear separated at most by $q$ units of time.

In this approach, it is rather important to observe the prominent role played by condition $\D'_q(u_{n,i})$. In particular, note that if condition $\D'_q(u_{n,i})$ holds for some particular $q=q_0\in\N_0$, then condition $\D'_q(u_{n,i})$ holds for all $q\geq q_0$. Then,  $q$  as defined in \eqref{def:q} is indeed the natural candidate to try to show the validity of $\D'_q(u_n)$.

We now give a way of defining the Extremal Index (EI) using the sets $\A_{n,i}$.
For $q\in\N_0$ given by \eqref{def:q}, we also assume that there exists $0\leq\theta\leq1$, which will be referred to as the EI, such that
\begin{equation}
\label{eq:EI}
\lim_{n\to\infty}\max_{i=1,\ldots,k_n}\left\{\left|\theta k_n\sum_{j=\L_{n,i-1}}^{\L_{n,i}-1}\bar F(u_{n,j})-k_n\sum_{j=\L_{n,i-1}}^{\L_{n,i}-1}\p\left(A_{n,j}^{(q)}\right)\right|\right\}=0.
\end{equation}

The following is the main theorem of this section.

\begin{theorem}
\label{thm:error-terms-general-SP-no-clustering}
Let $X_0, X_1, \ldots$ be a stationary stochastic process and suppose \eqref{Fmax} and \eqref{F_n} hold for some $\tau>0$. Let $q\in\N_0$ be as in \eqref{def:q} and assume that \eqref{eq:EI} holds. Assume also that conditions $\D(u_{n,i})$ and $\D'_q(u_{n,i})$ are satisfied. Then
\begin{align*}
\lim_{n\to \infty}\mathbf P_n=e^{-\theta \tau}.
\end{align*}
\end{theorem}

The rest of this section is devoted to the proof of Theorem~\ref{thm:error-terms-general-SP-no-clustering}.

To simplify notation, we will drop the index $n\in\N$ and write: $u_i:=u_{n,i}$, $A_{i}^{(q)}:=A_{n,i}^{(q)}$, $\mathbb A^{(q)}:=\mathbb A_n^{(q)}$, $\ell_{i}:=\ell_{n,i}$, $\L_{i}:=\L_{n,i}$.

\subsection{Preliminaries to the argument}

We begin by proving the crucial observation stated in Proposition~\ref{prop:relation-balls-annuli-general}.
\begin{proof}[Proof of Proposition~\ref{prop:relation-balls-annuli-general}]
Since $A_r\subset B_r$, then clearly $\mathscr W_{0,n}(\mathbb B)\subset \mathscr W_{0,n} (\mathbb A)$. Hence, we have to estimate the probability of $\mathscr W_{0,n} (\mathbb A)\setminus \mathscr W_{0,n}(\mathbb B)$.

Let $x\in\mathscr W_{0,n} (\mathbb A)\setminus \mathscr W_{0,n}(\mathbb B)$. We will see that there exists $j\in\{1,\ldots, q\}$ such that $x\in B_{n-j}$. In fact, suppose that no such $j$ exists. Then let $\ell=\max\{i\in\{1,\ldots, n-1\}:\, x\in B_i\}$. Then, clearly, $\ell<n-q$. Hence, if $x\notin B_j$, for all $i=\ell+1,\ldots,n-1$, then we must have that $x\in A_\ell$ by definition of $A$. But this contradicts the fact that $x\in\mathscr W_{0,n} (\mathbb A)$. Consequently, we have that there exists $j\in\{1,\ldots, q\}$ such that $x\in B_{n-j}$ and since $x\in\mathscr W_{0,n} (\mathbb A)$ then we can actually write $x\in B_{n-j}\setminus A_{n-j}$.

This means that $\mathscr W_{0,n} (\mathbb A)\setminus \mathscr W_{0,n}(\mathbb B)\subset \bigcup_{j=1}^q (B_{n-j}\setminus A_{n-j})\cap \mathscr W_{0,n} (\mathbb A)$ and then
\begin{multline*}
\big|\p(\mathscr W_{0,n}(\mathbb B))-\p(\mathscr W_{0,n}(\mathbb A))\big|=\p(\mathscr W_{0,n} (\mathbb A)\setminus \mathscr W_{0,n}(\mathbb B))\\ \leq
\p\left(\bigcup_{j=1}^q (B_{n-j}\setminus A_{n-j})\cap \mathscr W_{0,n} (\mathbb A)\right)\leq\sum_{j=1}^{q} \p\left(\mathscr W_{0,n}(\mathbb A)\cap (B_{n-j}\setminus A_{n-j})\right),
\end{multline*}
as required.
\end{proof}

We prove next some lemmata that pave the way for Proposition~\ref{prop:main-estimate-1}, which is the cornerstone of the argument leading to the proof of Theorem~\ref{thm:error-terms-general-SP-no-clustering}

\begin{lemma}
\label{lem:time-gap-1}
For any fixed  $\mathbb A=(A_0,A_1,\ldots)$, $A_i\in \mathcal B$ for $i=0,1,\ldots$, and integers $a,s,t,m$, with  $a<s$, we have:
\begin{equation*}
\left|\p(\mathscr W_{a,s+t+m}(\mathbb A))-\p(\mathscr W_{a,s}(\mathbb A)\cap \mathscr W_{a+s+t,m}(\mathbb A))\right|\leq \sum_{j=s}^{s+t-1}\p(A_{a+j}).
\end{equation*}
 \end{lemma}

  \begin{proof}
\begin{align*}
\p(\mathscr W_{a,s}(\mathbb A)\cap \mathscr W_{a+s+t,m}(\mathbb A))-\p(\mathscr W_{a,s+t+m}(\mathbb A))&=\p(\mathscr W_{a,s}(\mathbb A)\cap \mathscr W_{a+s,t}^c(\mathbb A)\cap\mathscr W_{a+s+t,m}(\mathbb A)) \\
&\leq\p(\mathscr W_{a+s,t}^c(\mathbb A))=\p(\cup_{j=s}^{s+t-1} (A_{a+j}))\\
& \leq \sum_{j=s}^{s+t-1}\p(A_{a+j}).
\end{align*}
\end{proof}

\begin{lemma}
\label{lem:inductive-step-1}For any fixed  $\mathbb A=(A_0,A_1,\ldots)$, $A_i\in \mathcal B$ for $i=0,1,\ldots$, and integers $a,s,t,m$, with  $a<s$, we have:
\begin{multline*}
\left|\p(\mathscr W_{a,s}(\mathbb A)\cap \mathscr W_{a+s+t,m}(\mathbb A))-\p(\mathscr W_{a+s+t,m}(\mathbb A))\left(1-\sum_{j=0}^{s-1}\p(A_{a+j})\right)\right|\leq\\
\leq\left|\sum_{j=0}^{s-1}\p(A_{a+j})\p(\mathscr W_{a+s+t,m}(\mathbb A))-\sum_{j=0}^{s-1}\p(A_{a+j}\cap \mathscr  W_{a+s+t,m}(\mathbb A))\right| +\sum_{j=0}^{s-1}\sum_{i>j}^{s-1}\p(A_{a+i}\cap A_{a+j}).
\end{multline*}
 \end{lemma}

\begin{proof}
Observe that
\begin{align*}
\label{eq:triangular1}
&\left|\p(\mathscr W_{a,s}(\mathbb A)\cap \mathscr W_{a+s+t,m}(\mathbb A))-\p(\mathscr W_{a+s+t,m}(\mathbb A))(1-\sum_{j=0}^{s-1}\p(A_{a+j}))\right|
\\
&\leq \left|\sum_{j=0}^{s-1}\p(A_{a+j})\p(\mathscr W_{a+s+t,m}(\mathbb A))-\sum_{j=0}^{s-1}\p(A_{a+j}\cap \mathscr  W_{a+s+t,m}(\mathbb A))\right|
 \\
&+\left|\p(\mathscr W_{a,s}(\mathbb A)\cap \mathscr W_{a+s+t,m}(\mathbb A))-\p(\mathscr W_{a+s+t,m}(\mathbb A))+\sum_{j=0}^{s-1}\p(A_{a+j}\cap \mathscr W_{a+s+t,m}(\mathbb A))\right|.
\end{align*}

Regarding the second term on the right, we have
\begin{align*}
\p(\mathscr W_{a,s}(\mathbb A)\cap \mathscr W_{a+s+t,m}(\mathbb A))=\p(\mathscr W_{a+s+t,m}(\mathbb A))-\p(\mathscr W_{a,s}^c(\mathbb A)\cap \mathscr W_{a+s+t,m}(\mathbb A)).
\end{align*}

Now, since $\mathscr W_{a,s}^c(\mathbb A)\cap \mathscr W_{a+s+t,m}(\mathbb A)=\cup_{i=0}^{s-1}(A_{a+i}\cap \mathscr W_{a+s+t,m}(\mathbb A))$, we have
\[
\p(\mathscr W_{a,s}^c(\mathbb A)\cap \mathscr W_{a+s+t,m}(\mathbb A))\leq \sum_{i=0}^{s-1}(A_{a+i}\cap\p( \mathscr W_{a+s+t,m}(\mathbb A)))
\]
and so,
\[
0\leq \sum_{j=0}^{s-1}\p(A_{a+j}\cap \mathscr  W_{s+t,m}(\mathbb A))-\p(\mathscr W_{a,s}^c(\mathbb A)\cap \mathscr W_{a+s+t,m}(\mathbb A))\leq \sum_{j=0}^{s-1}\sum_{i>j}^{s-1}\p(A_{a+i}\cap A_{a+j}\cap  \mathscr W_{a+s+t,m}(\mathbb A))
\]
Hence, using these last computations we get:
\begin{align}
\nonumber&\Big|\p(\mathscr W_{a,s}(\mathbb A)\cap \mathscr W_{a+s+t,m}(\mathbb A))-\p(\mathscr W_{a+s+t,m}(\mathbb A))+\sum_{j=0}^{s-1}\p(A_{a+j}\cap \mathscr W_{a+s+t,m}(\mathbb A))\Big|
\\
\nonumber&=\Big|-\p(\mathscr W_{a,s}^c(\mathbb A)\cap \mathscr W_{a+s+t,m}(\mathbb A))+\sum_{j=0}^{s-1}\p(A_{a+j}\cap \mathscr W_{a+s+t,m}(\mathbb A))\Big|
\\
\nonumber&\leq \sum_{j=0}^{s-1}\sum_{i> j}^{s-1}\p(A_{a+i}\cap A_{a+j}\cap \mathscr W_{a+s+t,m}(\mathbb A))\nonumber
\\
&\leq \sum_{j=0}^{s-1}\sum_{i> j}^{s-1}\p(A_{a+i}\cap A_{a+j}).\nonumber
\end{align}
\end{proof}

\subsection{The construction of the blocks}
\label{subsec:blocks}

The construction of the blocks here, contrary to the stationary case, in which the blocks have equal size, is designed so that the expected number of exceedances in each block is the same. We follow closely the construction in \cite{H83,H86}.

For each $n\in\N$ we split the random variables $X_0, \ldots, X_{n-1}$ into $k_n$ initial blocks, where $k_n$ is given by \eqref{eq:kn-sequence}, of sizes $\ell_1, \ldots, \ell_{k_n}$ defined in the following way. Let as before $\L_i=\sum_{j=1}^i \ell_i$ and $\L_0=\ell_0=0$. Assume that $\ell_1, \ldots, \ell_{i-1}$ are already defined. Take $\ell_i$ to be the largest integer such that:
$$
\sum_{j=\L_{i-1}}^{\L_{i-1}+\ell_i-1}\bar F(u_{n,i})\leq \frac{F_n^*}{k_n}.
$$

The final working blocks are obtained by disregarding the last observations of each initial block, which will create a time gap between each final block. The size of the time gaps must be balanced in order to have at least a size $t_n^*$ but such that its weight on the average number of exceedances is negligible when compared to that of the final blocks. For that purpose we define
$$
\eps(n):=(t_n^*+1)\bar F_{max} \frac{k_n}{F_n^*}.
$$
Note that by \eqref{F_n} and \eqref{eq:kn-sequence}, it follows immediately that $\lim_{n\to\infty}\eps(n)=0$. Now, for each $i=1,\ldots,k_n$ let $t_i$ be the largest integer such that
$$
\sum_{j=\L_{i}-t_i}^{\L_i-1}\bar F(u_{n,i})\leq \eps(n)\frac{F_n^*}{k_n}.
$$

Hence, the final working blocks correspond to the observations within the time frame $\L_{i-1}+1,\ldots,\L_i-t_i$, while the time gaps correspond to the observations in the time frame $\L_i-t_i+1,\ldots, \L_i$, for all $i=1,\ldots,k_n$.

Note that $t_n^*\leq t_i< \ell_i$, for each $i=1,\ldots,k_n$.
 The second inequality is trivial. For the first inequality note that by definition of $t_i$ we have
$$ \eps(n)\frac{F_n^*}{k_n}\leq\sum_{j=\L_{i}-t_i}^{\L_i-1}\bar F(u_{n,i})+\bar F(u_{n,\L_i-t_i-1})\leq (t_i+1) \bar F_{max}.$$ The first inequality follows easily now by definition of $\eps(n)$.

\begin{propositionP}
\label{prop:main-estimate-1} For every, $n\in\N$, let $\mathbb A:=\mathbb A_n^{(q)}$ for $q$ defined by \eqref{def:q}. Consider the construction of the $k_n$ blocks above, the respective sizes $\ell_1, \ldots, \ell_{k_n}$ and time gaps $t_1,\ldots, t_{k_n}$. Recall that $\L_i=\sum_{j=1}^i\ell_i$. Assume that $n\in\N$ is large enough  so that $F_n^*/k_n<2$.  We have:
\begin{align*}
\nonumber \Bigg|\p\big(\mathscr W_{0,n}(\mathbb A)\big)&-\prod_{i=1}^{k_n}\left(1-\sum_{j=\L_{i-1}}^{\L_i-t_i-1}\p(A_{j})\right)\Bigg|\leq\sum_{i=1}^{k_n}\sum_{j=\L_{i-1}-t_i}^{\L_i-1}\p(A_j^{(q)})+  \sum_{j=\L_{k_n}}^{n-1}\p(A_j^{(q)}) \nonumber\\
&+\sum_{i=1}^{k_n}  \left|\sum_{j=0}^{\ell_i-t_i-1}\bigg(\p(A_{\L_{i-1}+j})\p(\mathscr W_{\L_i,\L_{k_n}-\L_{i}}(\mathbb A))-\p(A_{\L_{i-1}+j}\cap \mathscr W_{\L_i,\L_{k_n}-\L_{i}}(\mathbb A))\bigg)\right|\\
&+ \sum_{i=1}^{k_n} \sum_{j=0}^{\ell_i-1} \sum_{r>j}^{\ell_i-1}\p(A_{\L_{i-1}+j}\cap A_{\L_{i-1}+r}).
\end{align*}
 \end{propositionP}

\begin{proof}
Using Lemma~\ref{lem:time-gap-1}, we have:
\begin{equation}
\label{eq:approx1-1}
\Big|\p(\mathscr W_{0,n}(\mathbb A))-\p(\mathscr W_{0,\L_{k_n}}(\mathbb A))\Big|\leq\sum_{j=\L_{k_n}}^{n-1}\p(A_j^{(q)}).\end{equation}

To simplify the notation let $\bar\L_i=\L_{k_n}-\L_{i-1}=\sum_{j=i}^{k_n} \ell_{j}$.
 It follows by using  \eqref{lem:inductive-step-1} that
\begin{align}
\Bigg|\p\left(\mathscr W_{\L_{i-1},\bar\L_{i}}(\mathbb A)\right)-&\bigg(1-\sum_{j=\L_{i-1}}^{\L_i-t_i-1}\p(A_{j})\bigg)\p\left(\mathscr W_{\L_{i},\bar\L_{i+1}}(\mathbb A)\right)\Bigg|\nonumber\\&\leq
\left|\p(\mathscr W_{\L_{i-1},\bar\L_{i}}(\mathbb A))-\p(\mathscr W_{\L_{i-1},\ell_i-t_i}(\mathbb A)\cap\mathscr W_{\L_{i},\bar\L_{i+1}}(\mathbb A))\right|\nonumber\\
&\quad+\left|\p(\mathscr W_{\L_{i-1},\ell_i-t_i}(\mathbb A)\cap\mathscr W_{\L_{i},\bar\L_{i+1}}(\mathbb A))-\Big(1-\sum_{j=\L_{i-1}}^{\L_i-t_i-1}\p(A_{j})\Big)\p(\mathscr W_{\L_{i},\bar\L_{i+1}}(\mathbb A))\right|\nonumber \\
&\leq \sum_{j=\L_{i-1}-t_i}^{\L_i-1}\p(A_j)+\left|\sum_{j=\L_{i-1}}^{\L_i-t_i-1}(\p(A_{j})\p(\mathscr W_{\L_{i-1},\bar\L_{i}}(\mathbb A))-\p(A_{j}\cap \mathscr W_{\L_{i-1},\bar\L_{i}}(\mathbb A))\right| \nonumber \\
&\quad+\sum_{j=0}^{\ell_i-1} \sum_{r>j}^{\ell_i-1}\p(A_{\L_{i-1}+j}\cap A_{\L_{i-1}+r}).\label{eq:induction-step-1}
\end{align}
Let
\begin{align}
\Upsilon_i&:=\sum_{j=\L_{i-1}-t_i}^{\L_i-1}\p(A_j)+\left|\sum_{j=\L_{i-1}}^{\L_i-t_i-1}(\p(A_{j})\p(\mathscr W_{\L_{i-1},\bar\L_{i}}(\mathbb A))-\p(A_{j}\cap \mathscr W_{\L_{i-1},\bar\L_{i}}(\mathbb A))\right|
\nonumber\\
&\qquad+\sum_{j=0}^{\ell_i-1} \sum_{r>j}^{\ell_i-1}\p(A_{\L_{i-1}+j}\cap A_{\L_{i-1}+r}).
\nonumber
\end{align}
Note that, for $i=k_n$ in (\ref{eq:induction-step-1}),  $\left|\mathscr W_{\L_{k_n-1},\bar\L_{k_n}}(\mathbb A))-\bigg(1-\sum_{j=\L_{k_n-1}}^{\L_{k_n}-t_{k_n}-1}\p(A_{j})\bigg)\right|\leq \Upsilon_{k_n}$.

Since $\frac{F_n^*}{k_n}<2$ and, by construction, for all $i=1,\ldots, k_n$, it is clear that $\sum_{j=\L_{i-1}}^{\L_i-t_i-1}\p(A_{j})\leq\frac{F_n^*}{k_n}$, then  $\left|1-\sum_{j=\L_{i-1}}^{\L_i-t_i-1}\p(A_{j})\right|<1$, for all $i=1,\ldots,k_n$.

Now, we use (\ref{eq:induction-step-1}) recursively and obtain
\begin{align}
\label{eq:recursive-estimate-1}
\left|\p(\mathscr W_{0,\L_{k_n}}(\mathbb A))-\prod_{i=1}^{k_n}\bigg(1-\sum_{j=\L_{i-1}}^{\L_i-t_i-1}\p(A_{j})\bigg)\right|\leq \sum_{i=1}^{k_n}\Upsilon_i.
\end{align}
The result follows now at once from (\ref{eq:approx1-1}) and (\ref{eq:recursive-estimate-1}).
\end{proof}

\subsection{Final argument}

We are now in a position to prove Theorem~\ref{thm:error-terms-general-SP-no-clustering}.

\begin{proof}[Proof of Theorem~\ref{thm:error-terms-general-SP-no-clustering}]

The theorem follows if we show that all the error terms in Proposition~\ref{prop:main-estimate-1} converge to 0, as $n\to\infty$.

For the first term, by choice of the $t_i$'s, we have
$$
\sum_{i=1}^{k_n}\sum_{j=\L_{i-1}-t_i}^{\L_i-1}\p(A_j^{(q)})\leq \sum_{i=1}^{k_n}\sum_{j=\L_{i-1}-t_i}^{\L_i-1}\bar F(u_{n,j})\leq k_n\eps(n)\frac{F_n^*}{k_n}=\eps(n)F_n^*,
$$
which tends to $0$ as $n\to\infty$, by (\ref{F_n}) and definition of $\eps(n)$.

Regarding the second term observe first that $$\sum_{j=\L_{k_n}}^{n-1}\p(A_j^{(q)})\leq \sum_{j=\L_{k_n}}^{n-1}\bar F(u_{n,j}).$$

Since, by choice of  $\ell_i$, we have $\frac{F_n^*}{k_n}\leq \sum_{j=\L_{i-1}}^{\L_i-1}\bar F(u_{n,j})+\bar F(u_{n,\L_i})\leq
 \sum_{j=\L_{i-1}}^{\L_i-1}\bar F(u_{n,j})+\bar F_{max}$, then it follows that
\begin{equation}
\label{eq:block-estimate}
 \frac{F_n^*}{k_n}-\bar F_{max}\leq\sum_{j=\L_{i-1}}^{\L_i-1}\bar F(u_{n,j})\leq \frac{F_n^*}{k_n}.
 \end{equation}
 From the first inequality we get $F_n^*-k_n\bar F_{max}\leq\sum_{i=1}^{k_n}\sum_{j=\L_{i-1}}^{\L_i-1}\bar F(u_{n,j})$, which implies that
 $$
 \sum_{j=\L_{k_n}}^{n-1}\bar F(u_{n,j})=F_n^*-\sum_{i=1}^{k_n}\sum_{j=\L_{i-1}}^{\L_i-1}\bar F(u_{n,j})\leq k_n\bar F_{max},$$
which goes to $0$ as $n\to\infty$ by \eqref{eq:kn-sequence}.

For the third term, recalling that, for each $n$ and $i$, $\gamma_i(q,n,t)$ from condition $\D_q(u_{n,i})$ is decreasing in $t$, we have:
\begin{equation*}
\sum_{i=1}^{k_n}  \left|\sum_{j=0}^{\ell_i-t_i-1}\bigg(\p(\A_{\L_{i-1}+j})\p(\mathscr W_{\L_i,\L_{k_n}-\L_{i}}(\mathbb A))-\p(\A_{\L_{i-1}+j}\cap \mathscr W_{\L_i,\L_{k_n}-\L_{i}}(\mathbb A))\bigg)\right|\leq \sum_{i=0}^{n-1}\gamma_i(q,n,t_n)\nonumber,
\end{equation*}
which tends to $0$ as $n\to\infty$ by condition $\D_q(u_{n,i})$.

By condition $\D'(u_n)$, we have that the fourth term goes to $0$ as $n\to\infty$.

Now, we will see that
\begin{align}
\nonumber&\left|\prod_{i=1}^{k_n}\left(1-\sum_{j=\L_{i-1}}^{\L_i-t_i-1}\p(\A_{j})\right)-e^{-\theta\tau}\right|\xrightarrow[n\to\infty]{} 0.
\end{align}

By \eqref{eq:EI} we have that $k_n\sum_{j=\L_{i-1}}^{\L_i-1}\p(\A_{j})=k_n \theta \sum_{j=\L_{i-1}}^{\L_i-1}\bar F(u_{n,j})+o(1)$. Then
\begin{equation*}
\sum_{j=\L_{i-1}}^{\L_i-1}\p(\A_{j})=\theta\sum_{j=\L_{i-1}}^{\L_i-1}\bar F(u_{n,j})+o(k_n^{-1}).
\end{equation*}
Since by \eqref{eq:kn-sequence}, we have $\bar F_{max}=o(k_n^{-1})$, then, by \eqref{eq:block-estimate}, it follows that
\begin{equation*}
\sum_{j=\L_{i-1}}^{\L_i-1}\bar F(u_{n,j})+o(k_n^{-1})=\frac{F_n^*}{k_n}+o(k_n^{-1}).
\end{equation*}
Also note that
\begin{equation*}
\sum_{j=\L_{i}-t_i}^{\L_i-1}\p(\A_{j})\leq \sum_{j=\L_{i}-t_i}^{\L_i-1}\bar F(u_{n,j})\leq \eps(n)\frac{F_n^*}{k_n}=o(k_n^{-1}).
\end{equation*}
Hence, for all $i=1,\ldots,k_n$ we have
\begin{equation*}
\sum_{j=\L_{i}-t_i}^{\L_i-t_i-1}\p(\A_{j})=\theta \frac{F_n^*}{k_n}+o(k_n^{-1}).
\end{equation*}
Finally, by \eqref{F_n}, we have
\begin{equation*}
\prod_{i=1}^{k_n}\left(1-\sum_{j=\L_{i}-t_i}^{\L_i-t_i-1}\p(\A_{j})\right)\sim\left(1-\theta\frac{F_n^*}{k_n}+o(k_n^{-1})\right)^{k_n}\xrightarrow[n\to\infty]{}\e^{-\theta \tau}.
\end{equation*}

Finally, by Proposition~\ref{prop:relation-balls-annuli-general}  we have
\begin{align}
\label{eq:error3}
\left|P_n-\p\left(\mathscr W_{0,n}\left(\mathbb A^{(q)}\right)\right)\right|&\leq\sum_{j=1}^{q} \p\left(\mathscr W_{0,n}\left(\mathbb A^{(q)}\right)\cap \left(\{X_{n-j}> u_{n,n-j}\}\setminus\{A_{n-j}^{(q)}\}\right)\right)\nonumber\\
&\leq\sum_{j=1}^{q} \p \left(\{X_{n-j}> u_{n,n-j}\}\setminus\{A_{n-j}^{(q)}\}\right)\nonumber\\
&\leq\sum_{j=1}^{q} (1-F_{n-j}(u_{n,n-j})),
\end{align}
which converges to $0$ as $n\to\infty$.

Note that when $q=0$ both sides of inequality \eqref{eq:error3} equal 0.
\end{proof}

\section{Sequential Dynamical Systems}%: general considerations}
\subsection{General presentation}
In this section we will give a first example of a non-stationary process, by considering  families ${\mathcal F}$ of non-invertible maps  defined
on compact subsets $X$ of $\mathbb{R}^d$ or on the torus $\mathbb{T}^d$
(still denoted with $X$ in the following),  and non-singular with respect
to the Lebesgue or the Haar measure,
i.e. $m(A) \not= 0 \implies m(T(A)) \not= 0$.
Such measures will be defined on the Borel sigma algebra ${\mathcal B}$. We will be mostly concerned with the case $d=1.$  A countable sequence of maps $\{T_k\}_{k\ge 1}\in \mathcal{F}$  defines a {\em sequential dynamical system}. A {\em sequential orbit} of $x\in X$  will be defined by the concatenation
\begin{equation}\label{m}
\TF_n (x) :=T_n\circ\cdots\circ T_1 (x), \ n\ge 1.
\end{equation}
We  denote by $P_{j}$ the Perron-Fr\"obenius (transfer) operator associated to $T_{j}$ defined by the duality relation
$$
\int_X P_{j}f \ g\ dm\ = \ \int_X f\ g\circ T_{j} \ dm,\;\; \mbox{ for all } f\in L^1_m, \ g\in L^{\infty}_m.
$$
Note that here the transfer operator $P_{j}$ is defined with respect to the reference Lebesgue measure $m$.

Similarly to (\ref{m}), we define the composition of operators as
\begin{equation}\label{o}
\Pi_n:=P_n\circ\cdots\circ P_1, \ n\ge 1.
\end{equation}
It is easy to check that duality persists under concatenation, namely
\begin{equation}\label{c}
\int_X g  (\TF_n) \  f \ dm=\int_X g (T_n\circ\cdots\circ T_1) \ f\ dm\ =\ \int_X g(\ P_n\circ\cdots\circ P_1f )\ dm\ =  \int_X g\ (\Pi_n f) \ dm.
\end{equation}

In \cite{CR07} the authors begin a systematic study of the statistical properties of sequential dynamical systems by proving in particular the law of large numbers and the central limit theorem. In \cite{HNTV}, it was shown that the Almost Sure Invariance Principle still holds. In order to establish such results a few assumptions are needed and some of them are also relevant for the extreme value theory. We will recall them in this section and then we will provide a list of examples which will go beyond the $\beta$ transformations, which was the prototype case investigated by Conze and Raugi.  \\

We first need to
choose a suitable couple of adapted spaces in order to get and exploit the quasi-compactness of the transfer operator.  We will consider in particular  a Banach space ${\mathcal V} \subset L^1_m$
($1\in\mathcal{V}$)
of functions over $X$ with norm $||\cdot||_{\alpha}$, such that
$\|\phi\|_{\infty} \le C \|\phi\|_{\alpha}$.

 For example, we could let  ${\mathcal V} $ be the Banach space  of bounded variation functions over $X$ with norm $||\cdot||_{BV}$ given by the sum of the $L^1_m$ norm and the total variation $|\cdot|_{BV},$ or we could take ${\mathcal V}$ to be the space of quasi-H\"older functions with a suitable norm which we will define later on.\\
One of the basic assumption is the following:\\

{\bf Uniform Doeblin-Fortet-Lasota-Yorke inequality (DFLY)}: There exist constants $A, B<\infty, \rho\in(0,1),$ such that for any $n$ and any sequence of operators $P_n,\cdots,P_1$ associated to transformations in $\mathcal{F}$ and any $f\in {\mathcal V}$ we have
\begin{equation}\label{LY2}
\|P_n\circ \cdots\circ P_1 f\|_{\alpha}\le A \rho^n \|f\|_{\alpha}+B\|f\|_1.
\end{equation}\\
At this point one would like to dispose of a  sort of quasi-compactness argument which would allow to get exponential decay for the composition of operators. In all the examples  we will present, the class $\mathcal{F}$ will be constructed {\em around} (this will be made clear in a moment)  a given map $T_0$ for which the corresponding operator $P_0$ will satisfy quasi-compactness. Namely we require:

{\bf Exactness property}: The operator $P_0$ has a spectral gap, which
implies that there are two constants $C_1<\infty$ and $\gamma_0\in(0,1)$
so that
\begin{equation}\label{Exa}
 ||P_0^n f||_{\alpha}\le C_1 \gamma_0^n||f||_{\alpha}
\end{equation} for all $f\in \mathcal{V}$ of zero (Lebesgue) mean and $n\ge 1$.\\

The next step is to  consider  the following distance between two operators $P$ and $Q$ associated to maps in $\mathcal{F}$ and acting on $\mathcal{V}$:
$$
d(P,Q)= \sup_{f\in \mathcal{V},\; \|f\|_{\alpha}\le 1}||P f-Q f||_1.
$$

A very useful criterion  is given in Proposition 2.10 in \cite{CR07}, and in our setting it reads: {\em if $P_0$ verifies the exactness property, then there exists $\delta_0>0,$ such that the set $\{P\in {\mathcal F};  d(P,P_0)<\delta_0\}$ satisfies the decorrelation  \emph{(DEC)} condition}, where

\noindent {\bf Property (DEC):}
Given the  family $\mathcal {F}$    there exist constants $\hat{C}>0, \hat{\gamma}\in(0,1),$ such that for any $n$ and any sequence of  transfer operators $P_n,\cdots,P_1$ corresponding to maps
chosen from $\mathcal {F}$ and any $f\in {\mathcal V}$ of zero (Lebesgue) mean\footnote{Actually, the definition of the (DEC) property in \cite{CR07} is slightly more general since it requires the above property for functions in a suitable subspace, not necessarily that of functions with zero expectation.}, we have
\begin{equation}\label{dec}
\|P_n\circ \cdots\circ P_1 f\|_{\alpha}\le \hat{C}\hat{\gamma}^n \|f\|_{\alpha}.
\end{equation}

By
induction on the Doeblin-Fortet-Lasota-Yorke inequality for compositions we immediately
have
\begin{equation}\label{VC}
 d(P_r \circ \cdots \circ P_1, P_0^r)\le M \sum_{j=1}^r d(P_j, P_0),
\end{equation}
with $M=1+A\rho^{-1}+B.$

According to~\cite[Lemma~2.13]{CR07}, (\ref{Exa}) and (\ref{VC}) imply that there exists a constant
$C_2$ such that
$$
\|P_n\circ\cdots\circ P_1\phi-P_0^n\phi\|_{1}\le C_2\|\phi\|_{BV}\left(\sum_{k=1}^pd(P_{n-k+1},P_0)+(1-\gamma_0)^{-1}\gamma_0^p\right)
$$
 for all
integers $p\le n$ and all functions $\phi\in \mathcal{V}$. We will use this bound to get a quantitative rate of the exponential decay for composition of  operators in the $L^1_m$ norm when we relate it to the following two assumptions: \\

{\bf  Lipschitz continuity property}: Assume that the maps (and their transfer
operators) are parametrized by a sequence of numbers $\eps_k$, $k\in
\mathbb{N}$, such that $\lim_{k\to\infty}\eps_k=\eps_0$ ($P_{\eps_0}=P_0$).
We assume that there exists a constant $C_3$ so that
$$
 d(P_{\eps_k},P_{\eps_j})\le C_3|\eps_k-\eps_j|,\qquad \text{ for all $k, j\ge 0$}.
$$
We will restrict in the following to the subclass ${\mathcal F}_{exa}$ of maps, and therefore of operators, for which
$$
{\mathcal F}_{exa}:=\{P_{\eps_k}\in {\mathcal F};\  |\eps_k-\eps_0|<C_3^{-1}\delta_0\}.
$$
The maps in ${\mathcal F}_{exa}$ will therefore verify the (DEC) condition, but we will sometimes need something stronger, namely:  \\

{\bf Convergence property}: We require algebraic convergence of the
parameters, that is, there exist a constant $C_4$ and $\kappa>0$ so that
$$
 |\eps_n-\eps_0|\le \frac{C_4}{n^{\kappa}}\qquad \forall n\ge1.
$$

With these last assumptions, we get a polynomial decay for (\ref{VC}) of the
type $O(n^{-\kappa})$ and in particular we obtain the same algebraic
convergence in $L^1_m$ of $P_n\circ\cdots\circ P_1\phi$ to $h \int \phi \,dm$,
where $h$ is the density of the absolutely continuous mixing measure of the
map $T_0.$

\subsection{Stochastic processes for sequential systems}
\label{subset:observable-choice}
Similarly to \cite{FFT10} (in the context of stationary deterministic systems), we consider that the time series $X_0, X_1,\ldots$ arises from these sequential systems simply by evaluating a given observable $\varphi:X\to\R\cup\{\pm\infty\}$ along the sequential orbits.
\begin{equation}
\label{eq:def-stat-stoch-proc-DS} X_n=\varphi\circ \TF_n,\quad \mbox{for
each } n\in {\mathbb N}.
\end{equation}
Note that, contrary to the setup in \cite{FFT10},   the stochastic process $X_0, X_1,\ldots$ defined in this way is not necessarily stationary.

We assume that the r.v. $\varphi:X\to\R\cup\{\pm\infty\}$
achieves a global maximum at $\zeta\in X$ (we allow
$\varphi(\zeta)=+\infty$) being of following form:
\begin{equation}
\label{eq:observable-form}
\varphi(x)=g\big(\dist(x,\zeta)\big),
\end{equation} where $\zeta$ is a chosen point in the
phase space $X$ and the function $g:[0,+\infty)\rightarrow {\mathbb
R\cup\{+\infty\}}$ is such that $0$ is a global maximum ($g(0)$ may
be $+\infty$); $g$ is a strictly decreasing bijection $g:V \to W$
in a neighbourhood $V$ of
$0$; and has one of the
following three types of behaviour:
\begin{enumerate}
\item[Type $g_1$:] there exists some strictly positive function
$h:W\to\R$ such that for all $y\in\R$
\begin{equation}\label{eq:def-g1}\displaystyle \lim_{s\to
g_1(0)}\frac{g_1^{-1}(s+yh(s))}{g_1^{-1}(s)}=\e^{-y};
\end{equation}
\item [Type $g_2$:] $g_2(0)=+\infty$ and there exists $\beta>0$ such that
for all $y>0$
\begin{equation}\label{eq:def-g2}\displaystyle \lim_{s\to+\infty}
\frac{g_2^{-1}(sy)}{g_2^{-1}(s)}=y^{-\beta};\end{equation}
\item [Type $g_3$:] $g_3(0)=D<+\infty$ and there exists $\gamma>0$ such
that for all $y>0$
\begin{equation}\label{eq:def-g3}\lim_{s\to0}
\frac{g_3^{-1}(D-sy)}{g_3^{-1}(D-s)}= y^\gamma.
\end{equation}
\end{enumerate}

It may be shown that no non-degenerate limit applies if $\int _0^{g_1(0)}g_1^{-1}(s)ds$ is not finite. Hence, an appropriate choice of $h$ in the Type 1 case is given by $h(s)=\int_ s^{g_1(0)} g_1^{-1}(t) dt / g_1^{-1}(s)$ for $s<g_1(0)$.

Examples of each one of the three types are as follows:
$g_1(x)=-\log x$ (in this case \eqref{eq:def-g1} is easily verified
with $h\equiv1$), $g_2(x)=x^{-1/\alpha}$ for some $\alpha>0$ (condition
\eqref{eq:def-g2} is verified with $\beta=\alpha$) and
$g_3(x)=D-x^{1/\alpha}$ for some $D\in\R$ and $\alpha> 0$ (condition
\eqref{eq:def-g3} is verified with $\gamma=\alpha$).

 \subsection{Examples}
 We now give a few examples of sequential systems  satisfying the preceding  assumptions. The family of maps $\mathcal{F}$ will be parametrized by a small positive number
 $\varepsilon$ (or a vector with small positive components) and  we will tacitly suppose that we restrict to ${\mathcal F}_{exa}$ having previously proved that the transfer operator $P_0$ for a reference map $T_0$ is exact. This will impose restrictions on the choice of $\varepsilon$ (less than a constant  times $\delta_0,$ see above), and in this case we will use the terminology {\em for $\varepsilon$ small enough}. The verification of the DFLY condition, which in turn will imply the analogous condition for the unperturbed operator $P_0$ will usually follow from standard arguments and  the exactness of $P_0$ will be proved by assuming the existence of a unique mixing absolutely continuous invariant measure (for instance by adding further properties to the map $T_0$), or alternatively by restricting to one of the finitely many mixing components prescribed by the quasi-compactness of $P_0.$\\

 The following examples have already been introduced and treated in \cite{HNTV}, but in the latter paper a  much stronger condition was required, namely that
 there exists $\delta>0$ such that for any sequence  $P_n,\cdots,P_1$ in $\mathcal{F}$  we have
 the uniform lower bound
\begin{equation}\label{MIN}
\inf_{x\in M}P_n\circ \cdots\circ P_11(x)\ge \delta,  \quad \forall n\ge 1.
\end{equation}
We do not need that property in the context of EVT.
 \subsubsection{$\beta$ transformation}
 Let $\beta>1$ and denote by $T_{\beta}(x)=\beta x$ mod $1$ the
$\beta$-transformation on the unit circle. Similarly, for $\beta_k\ge 1+c > 1$, $ k=1,2,\dots$,
we have the transformations $T_{\beta_k}$ of the same kind, $x\mapsto \beta_kx$ mod $1$.
Then ${\mathcal F}=\{T_{\beta_k}:k\}$ is the family of transformations we want to consider here.
The property~(DEC) was proved in~\cite[Theorem 3.4~(c)]{CR07}and continuity~(Lip) is
precisely the content of Sect.~5 still in \cite{CR07}.
\subsubsection{Random additive noise}

In this second example we consider   piecewise uniformly expanding maps $T$ on the unit interval $M=[0,1]$ which preserve a unique absolutely continuous invariant measure $\mu$ which is also mixing. We denote by $A_k, k=1,\dots,m$ the $m$ open intervals of monotonicity of the map $T$ which give a partition mod-$0$ of the unit interval.  The map $T$ is $C^2$ over the $A_k$ and with a $C^2$ extension on the boundaries. We put  $\min_{x\in M} |DT(x)|\ge \lambda>1;   \max_{x\in M} |DT(x)|\le \Lambda; \ \sup_{x\in M}\left|\frac{D^2T_{\eps}(x)}{DT_{\eps}(x)}\right|\leq C_1<\infty.$ We will perturb  with additive noise, namely  we will consider a family of maps $\mathcal{F}$ given by $T_{\eps}(x)=T(x)+\eps,$ where $\eps \in U$ and such that  $\forall \eps \in U$ we have the images $T_{\eps}A_k, k=1,\dots,m$  {\em strictly} included in $[0,1].$ We will also suppose that $\exists A_w$ such that $\forall T_{\eps}\in {\mathcal F}$ and $k=1,\dots,m: T_{\eps}A_k\supset A_w;$ moreover there exists $1\ge L'>0$ such that $\forall k=1,\dots,m$ and $\forall T_{\eps}\in {\mathcal F},$ $|T_{\eps}(A_w)\cap A_k|>L'.$
 These conditions are useful in obtaining  distortion bounds.  We note that our assumptions are satisfied if we consider $C^2$ uniformly expanding  maps on the circle and again perturbed with additive noise, without, this time, any restriction of the values of $\eps.$ In particular,  the intervals of local injectivity
$A_k, \ k=1,\cdots,m,$ of $T_{\eps}$ are now independent of $\eps$. The functional space $\mathcal{V}$ will coincide with the functions of bounded variation with norm $||\cdot||_{BV}.$\\
The (DFLY) inequality follows easily with standard arguments.%, see for instance \cite{AB}.
  The next step is to show that two operators are close when the relative perturbation parameters are close: we report here for completeness the short proof already given in \cite{HNTV}.  We thus consider the difference $||\hat{P}_{\eps_1}f-\hat{P}_{\eps_2}f||_1,$ with $f$ in BV. We have
  $$
 \hat{P}_{\eps_1}f(x)-\hat{P}_{\eps_2}f(x)=\sum_{l=1}^m f\cdot{\bf 1}_{U^c_n}(T^{-1}_{\eps_1,l}x)\left[\frac{1}{DT_{\eps_1}(T^{-1}_{\eps_1,l}x)}-\frac{1}{DT_{\eps_2}(T^{-1}_{\eps_2,l}x)}\right]+
 $$
 $$
 \sum_{l=1}^m\frac{1}{DT_{\eps_2}(T^{-1}_{\eps_2,l}x)}[f\cdot{\bf 1}_{U^c_n}(T^{-1}_{\eps_1,l}x)-f\cdot{\bf 1}_{U^c_n}(T^{-1}_{\eps_2,l}x)]=E_2(x)+E_3(x).
 $$
 In the formula above we considered, without restriction,  the derivative positive and moreover we discarded those points $x$ which have only one pre-image in each interval of monotonicity. After integration this will give an error $(E_1)$  as $E_1\le  4m|\eps_1-\eps_2|||\hat{P}_{\eps}f||_{\infty}.$ But $||\hat{P}_{\eps}f||_{\infty}\le ||f||_{\infty}\sum_{l=1}^m\frac{DT_{\eps_2}(T^{-1}_{\eps_2,l}x')}{DT_{\eps_2}
 (T^{-1}_{\eps_2,l}x)}\frac{1}{DT_{\eps_2}(T^{-1}_{\eps_2,l}x')},$
 where  $x'$ is the point  where $DT_{\eps_2}(T^{-1}_{\eps_2,l}x')|A_l|\ge \eta$, being $\eta$ the minimum length of $T(A_k), k=1,\dots,m.$ But the first ratio in the previous sum is simply bounded by the distortion constant $D_c=\Lambda\lambda^{-1}$; therefore we get
 $$
 E_1\le 4m|\eps_1-\eps_2| ||f||_{\infty} \frac{D_c}{\eta} \sum_{l=1}^m |A_l|\le  4m|\eps_1-\eps_2| ||f||_{\infty} \frac{D_c}{\eta}.
 $$

We now bound $E_2.$ The term in the square bracket and for given $l$ (we drop this index in the derivatives in  the next formulas), will be equal to $\frac{D^2T(\xi)}{[DT(\xi)]^2}|T^{-1}_{\eps_1}(x)-T^{-1}_{\eps_2}(x)|$, being $\xi$ a point in the interior of $A_l.$ The first factor is uniformly bounded by $C_1.$ Since $x=T_{\eps_1}(T^{-1}_{\eps_1}(x))=T((T^{-1}_{\eps_1}(x))+\eps_1=T((T^{-1}_{\eps_2}(x))+\eps_2=
T_{\eps_2}(T^{-1}_{\eps_2}(x)),$ we have that $|T^{-1}_{\eps_1}(x)-T^{-1}_{\eps_2}(x)|=|\eps_1-\eps_2||DT(\xi')|^{-1},$ where $\xi'$ is in $A_l.$ Replacing  $\xi'$ by $T^{-1}_{\eps_1,l}x$, because of distortion, we get
$$
\int |E_2(x)| dx\le |\eps_1-\eps_2| C_1 D_c\int \left[\sum_{l=1}^m |f(T^{-1}_{\eps_1,l})|\frac{1}{DT_{\eps_1}(T^{-1}_{\eps_1,l}x)}\right]dx=
$$
$$
|\eps_1-\eps_2| C_1 D_c\int P_{\eps_1}(|f|)(x) dx=|\eps_1-\eps_2| C_1 D_c||f||_1.
$$
To bound the last term we use the formula (3.11), in \cite{CR07},
$$
\int \sup_{|y-x|\le t}|f(y)-f(x)|dx\le 2t\mbox{Var}(f),
$$
by observing again that $|T^{-1}_{\eps_1}(x)-T^{-1}_{\eps_2}(x)|=|\eps_1-\eps_2||DT(\xi')|^{-1},$ where $\xi'$ is in $A_l.$ By integrating $E_3(x)$ we get
$$
\int |E_{3}(x)|dx\le \ 2m \lambda^{-2}\  |\eps_1-\eps_2|\mbox{Var}(f{\bf 1}_{U^c_n})\le
$$
$$
 10m \lambda^{-2}\  |\eps_1-\eps_2|\mbox{Var}(f).
$$
Putting together the three errors we finally get that there exists a constant $\tilde{C}$ such that
$$
||\hat{P}_{\eps_1}f-\hat{P}_{\eps_2}f||_1\le \tilde{C} |\eps_1-\eps_2|||f||_{BV},
$$
and we can complete the argument as in  the first example of $\beta$ transformations.
\subsubsection{Multidimensional maps}
 We give here a multidimensional version of the maps
considered in the preceding section; these maps were extensively
investigated in \cite{S00, HV09a, AFV15, AFLV11, HNVZ13}  and we defer to those papers
for more details. Let $M$ be a compact subset of $\mathbb{R}^N$
which is the closure of its non-empty interior. We take a map $T:M\to M$
and let $\mathcal{A}=\{A_i\}_{i=1}^m$ be a finite family of disjoint open sets such
that the Lebesgue measure of $M\setminus\bigcup_{i}A_i$ is zero, and there
exist open sets $\tilde{A_i}\supset\overline{A_i}$ and $C^{1+\alpha}$ maps
$T_i: \tilde{A_i}\to \mathbb{R}^N$, for some real number $0<\alpha\leq 1$
and some sufficiently small real number $\eps_1>0$, such that
\begin{enumerate}
\item $T_i(\tilde{A_i})\supset B_{\eps_1}(T(A_i))$ for each $i$, where
  $B_{\eps}(V)$ denotes a neighborhood of size $\eps$ of the set $V.$ The
  maps $T_i$ are the local extensions of $T$ to the $\tilde{A_i}.$

\item there exists a constant $C_1$ so that for each $i$ and $x,y\in T(A_i)$ with
$\mbox{dist}(x,y)\leq\eps_1$,
$$
|\det DT_i^{-1}(x)-\det DT_i^{-1}(y)|\leq C_1|\det DT_i^{-1}(x)|\mbox{dist}(x,y)^\alpha;
$$

\item there exists $s=s(T)<1$ such that $\forall x,y\in T(\tilde{A}_i) \textrm{ with } \mbox{dist}(x,y)\leq\eps_1$, we have
$$\mbox{dist}(T_i^{-1}x,T_i^{-1}y)\leq s\, \mbox{dist}(x,y);$$

\item each $\partial A_i$ is a codimension-one embedded compact
  piecewise $C^1$ submanifold and
  \begin{equation}\label{sc}s^\alpha+\frac{4s}{1-s}Z(T)\frac{\gamma_{N-1}}{\gamma_N}<1,\end{equation}
  where $Z(T)=\sup\limits_{x}\sum\limits_i \#\{\textrm{smooth pieces
    intersecting } \partial A_i \textrm{ containing }x\}$ and $\gamma_N$ is
  the volume of the unit ball in $\mathbb{R}^N$.
\end{enumerate}

Given such a map $T$, we define locally on each $A_i$ the map $T_{\eps}\in \mathcal{F}$ by
$T_{\eps}(x):=T(x)+\eps$, where now $\eps$ is an $n$-dimensional vector with
all the components of absolute value less than one. As in the previous
example the translation by $\eps$ is allowed if the image $T_{\eps}A_i$
remains in $M$: in this regard, we could play with the sign of the
components of $\eps$ or not move the map at all. As in the one dimensional case,
we shall also make the following  assumption on ${\mathcal F}$.
We assume that there exists a set $A_w\in \mathcal{A}$ satisfying:

\begin{itemize}
\item [(i)] $A_w\subset T_{\eps}A_k$ for all $\forall \;T_{\eps}\in{\mathcal F}$ and
  for all $k=1,\dots,m$.

\item [(ii)]  $TA_w$ is the whole $M$, which in turn implies that there exists $1\ge L'>0$ such that $\forall k=1,\dots,q$ and $\forall T_{\eps}\in {\mathcal F},$ $\mbox{diameter}(T_{\eps}(A_w)\cap A_k)>L'.$
\end{itemize}

As $\mathcal V\subset \mathscr{L}^1(m)$ we use the space of quasi-H\"older functions,
for which we refer again to \cite{S00, HV09a}. On this space, the transfer operator  satisfies a Doeblin-Fortet-Lasota-Yorke
  inequality.
  Finally, Lipschitz continuity has been proved for additive noise in Proposition~4.3 in~\cite{AFV15}.
\subsubsection{Covering maps: a general class}
We now present a more general class of examples which were introduced in~\cite{BV13} to study
metastability for randomly perturbed maps. As before, the family ${\mathcal F}$ will be constructed
 around a given map $T$ which is again defined on the unit interval $M$. We therefore begin by introducing such map $T$. \\
 {\bf (A1)} There exists a partition $\mathcal{A}=\{A_i:i=1,\dots,m\}$ of $M$, which consists
 of pairwise disjoint intervals $A_i$. Let $\bar A_i :=[c_{i,0},c_{i+1,0}]$. We assume there
 exists $\delta>0$  such that $T_{i,0}:= T|_{(c_{i,0},c_{i+1,0})}$ is $C^2$ and extends to a
  $C^2$ function $\bar T_{i,0}$ on a neighbourhood $[c_{i,0}-\delta,c_{i+1,0}+\delta]$ of $\bar A_i $ ;\\
 {\bf (A2)} There exists $\beta_0<\frac12$ so that $\inf_{x\in I\setminus {\mathcal C}_0}|T'(x)|\ge\beta_0^{-1}$, where
  ${\mathcal C}_0=\{c_{i,0}\}_{i=1}^{m}$.
 \\ We note that Assumption {\bf (A2)}, more precisely the fact that
 $\beta_0^{-1}$ is strictly bigger than $2$ instead of $1$, is sufficient
 to
 get the uniform Doeblin-Fortet-Lasota-Yorke inequality~\eqref{LLYY} below, as explained in Section 4.2 of \cite{GHW11}. We now construct the family ${\mathcal F}$ by choosing maps $T_{\eps}\in {\mathcal F}$ close to $T_{\eps=0}:=T$ in the following way:\\
 Each map $T_{\eps}\in {\mathcal F}$ has $m$ branches and there exists a partition of $M$
 into intervals $\{A_{i,\eps}\}_{i=1}^{m}$, $A_{i,\eps}\cap A_{j,\eps}=\emptyset$ for $i\not= j$, $\bar A_{i,\eps} :=[c_{i,\eps},c_{i+1,\eps}]$ such that\\
 \begin{itemize}
 \item [(i)] for each $i$ one has that
   $[c_{i,0}+\delta,c_{i+1,0}-\delta]\subset
   [c_{i,\eps},c_{i+1,\eps}]\subset [c_{i,0}-\delta,c_{i+1,0}+\delta]$;
   whenever $c_{1,0}=0$ or $c_{q+1},0=1$, we {\em do not move} them with
   $\delta$. In this way, we have established a one-to-one correspondence
   between the unperturbed and the perturbed extreme points of $A_i$ and
   $A_{i, \eps}$. (The quantity $\delta$ is from Assumption~(A1) above.)

 \item [(ii)] the map $T_{\eps}$ is locally injective over the closed
   intervals $\overline{A_{i,\eps}}$, of class $C^2$ in their interiors,
   and expanding with $\inf_x|T_{\eps}'x|>2$. Moreover there exists
   $\sigma>0$ such that $\forall T_{\eps} \in {\mathcal F}, \forall i=1,\cdots,
   m$ and $\forall x \in [c_{i,0}-\delta,c_{i+1,0}+\delta]\cap \overline{A_{i,\eps}}$ where $c_{i,0}$
   and $c_{i, \eps}$ are two (left or right) {\em corresponding points}, we
   have:
   \begin{equation}\label{C1}
     |c_{i,0}-c_{i, \eps}|\le \sigma
   \end{equation}
   and
   \begin{equation}\label{C2}
     |\bar{T}_{i,0}(x)-T_{i, \eps}(x)|\le \sigma.
   \end{equation}
 \end{itemize}

 Under these assumptions and by taking, with obvious notations, a
 concatenation of $n$ transfer operators, we have the uniform
 Doeblin-Fortet-Lasota-Yorke
 inequality, namely there exist $\eta\in(0,1)$ and $B<\infty$ such that, for all $f\in
 BV$, all $n$ and all concatenations of $n$ maps of ${\mathcal F}$, we have
\begin{equation}\label{LLYY}
||P_{\eps_n}\circ \cdots\circ P_{\eps_1}f||_{BV}\le \eta^n||f||_{BV}+B||f||_1.
\end{equation}
About the continuity~(Lip): looking carefully at the proof
  of the continuity for the expanding map of the intervals, one sees that it  extends
   to the actual case if one gets the following bounds:
\begin{equation}\label{BB}
\left.\begin{array}{r}|T^{-1}_{\eps_1}(x)-T^{-1}_{\eps_2}(x)|\\ |DT_{\eps_1}(x)-DT_{\eps_2}(x)|
\end{array}\right\}=O((|\eps_1-\eps_2|),
\end{equation}
where the point $x$ is in the same domain of injectivity of the maps $T_{\eps_1}$ and $T_{\eps_2}$, the comparison of the {\em same} functions and derivative in two {\em different} points being controlled   by the condition (\ref{C1}).  The bounds~\eqref{BB} follow easily by adding to~\eqref{C1}, \eqref{C2} the further assumptions that $\sigma=O(\eps)$ and requiring a continuity condition for
derivatives like~\eqref{C2} and with $\sigma$ again being of order $\eps$.

\section{EVT for the  sequential systems: an example of uniformly expanding map}
In this section, %and in the next one,
we will  give a detailed analysis of the application of the general result obtained in Section 2 to %two
a particular sequential system. It is constructed with $\beta$ transformations; similar approach and technique can be used to treat the other examples of sequential systems introduced above with suitable adaptations  and modifications.
 We  point out that in this example we will take $u_{n,i}=u_n$, where $(u_n)_{n\in\N}$ satisfies $n\mu(U_n)=n\mu(X_0>u_n)\to\tau$, as $n\to\infty$ for some $\tau>0$, where $\mu$ is the invariant measure of the original map $T_\beta$.\\

Consider the family of maps on the unit circle $S^1=[0,1]$, with  the identification $0\sim1$, given by $T_{\beta}(x)=\beta x$ mod $1$ for $\beta>1+c$, with $c>0$. Note that for many such $\beta$, we have that $T_\beta(1)\neq 1$ and, by the identification $0\sim 1$, this means that $T_\beta$ as a map on $S^1$ is not continuous at $\zeta=0\sim1$. For simplicity we assume that $T_\beta(0)=0$ but consider that the orbit of $1$ is still defined to be $T_\beta(1), T_\beta^2(1),\ldots$ although, strictly speaking, $1\sim0$ should be considered a fixed point. In what follows $m$ denotes Lebesgue measure on $[0,1]$.

\begin{theorem}
Consider an unperturbed map $T_\beta$ corresponding to some $\beta=\beta_0>1+c$, with invariant absolutely continuous probability $\mu=\mu_\beta$. Consider a sequential system acting on the unit circle and given by $\TF_n=T_{n}\circ\cdots\circ T_1$, where $T_i=T_{\beta_{i-1}}$, for all $i=1,\ldots,n$ and $|\beta_n-\beta|\leq n^{-\xi}$ holds for some $\xi>1$.  Let $X_1, X_2,\ldots$ be defined by \eqref{eq:def-stat-stoch-proc-DS}, where the observable function $\varphi$, given by \eqref{eq:observable-form}, achieves a global maximum at a chosen $\zeta\in[0,1]$. Let $(u_n)_{n\in\N}$ be such that $n\mu(X_0>u_n)\to\tau$, as $n\to\infty$ for some $\tau\geq 0$. Then, there exists $0<\theta\leq 1$ such that
$$\lim_{n\to\infty}m(X_0\leq u_n,X_1\leq u_n,\ldots,X_{n-1}\leq u_n)=\e^{-\theta\tau}.$$
The value of $\theta$ is determined by the behaviour of $\zeta$ under the original dynamics $T_\beta$, namely,
\begin{itemize}

\item If the orbit of $\zeta$ by $T_\beta$ never hits $0\sim1$ and $\zeta$ is periodic of prime period $p$ \footnote{$T_\beta^p(\zeta)=\zeta$ and $p$ is the minimum integer with such property} then $\theta=1-\beta^{-p}$;

\item If the orbit of $\zeta$ by $T_\beta$ never hits $0\sim1$ and $\zeta$ is not periodic then $\theta=1$;

\item If $\zeta=0\sim1$ and $1$ is not periodic \footnote{$T^n_\beta(1)\neq0\sim1$ for all $n$}, then $\theta=\frac{d\mu}{dm}(0)(1-\beta^{-1})+\frac{d\mu}{dm}(1)$

\item If $\zeta=0\sim1$ and $1$ is periodic of prime period $p$ then $\theta=\frac{d\mu}{dm}(0)(1-\beta^{-1})+\frac{d\mu}{dm}(1)(1-\beta^{-p})$.
\end{itemize}

\end{theorem}

 We remark that if the decay rate of  $|\beta_n-\beta|$ is slower than in the statement of the theorem then the observed extremal index for the sequential system at periodic points of the original dynamics may be $1$ as shown in  Section~\ref{subsec:counter-example}.

%\subsection{EVT for the $\beta$-transformation}
  \subsection{Preliminaries}
 As we said above, we let  $\mu$ denote the invariant measure of the original map $T_\beta$ and let $h=\frac{d\mu}{dm}$ be its density.

We assume throughout this subsection that there exists $\xi>1$ such that
\begin{equation}
\label{eq:beta-rate}
|\beta_{n}-\beta|\leq \frac1{n^\xi}.
\end{equation}
Also let $0<\gamma<1$ be such that $\gamma\xi>1$. In what follows $P$ denotes the transfer operator associated to the unperturbed map $T_\beta$. Recall that $\Pi_i=P_i\circ\ldots\circ P_1$, where $P_i$ is the transfer operator associated to $T_i=T_{\beta_i}$, while $P^i$ is the corresponding concatenation for the unperturbed map $T_\beta$.
Note that by \cite[Lemma~3.10]{CR07}, we have
\begin{equation}
\label{eq:PF-approximation}
\left\| \Pi_i(g)-\int g dm\; h\right\|_1\leq C_1\frac{\log i}{i^\xi}\|g\|_{BV}.
\end{equation}

Consider a measurable set $A\subset [0,1]$. Then
\begin{align*}
m(\TF_j^{-1}(A))&=\int \I_A\circ T_j\circ\ldots\circ T_1 dm =\int  \I_A \Pi_j(1) dm\\
&=\int \I_{A} hdm + \int \I_{A}(\Pi_j(1)-h)dm.
\end{align*}
 By \eqref{eq:PF-approximation}, if $j\geq n^\gamma$ (recall that $\gamma\xi>1$) then we have $\int |\Pi_j(1)-h| dm\leq C_1 \frac{\log i}{i^\xi}=o(n^{-1}),$ which allows us to write:
 \begin{equation}
 \label{eq:iterated-measure}
 m(\TF_j^{-1}(A))=\mu(A)+o(n^{-1}).
 \end{equation}

\subsubsection{Verification of  condition \eqref{F_n}, \ie $\lim_{n\to\infty}\sum_{i=0}^{n-1}m(X_i>u_n)=\tau$.}
We start with the following lemma.
\begin{lemma}
\label{lem:m-stationarity}
We have that
$$
\lim_{n\to\infty}\sum_{i=0}^{n-1}\int_{U_n} P^i(1) \,dm=\tau.
$$
\end{lemma}
\begin{proof}
By hypothesis, for all $j\in\N$ and $g\in BV$ we have $P^j(g)=h\int g \cdot h \,dm + Q^j(g)$, where $\|Q^j(g)\|_\infty\leq \alpha^j \|g\|_{BV}$, for some $\alpha<1$. Then we can write:
\begin{align*}
\sum_{i=0}^{n-1}\int_{U_n} P^i(1)dm&=\sum_{i=0}^{n-1}\int h \left(\int 1\cdot h dm\right) \1_{U_n} dm+\sum_{i=0}^{n-1} \int Q^i(1)\1_{U_n} dm\\
&=\sum_{i=0}^{n-1}\int_{U_n} h dm +\sum_{i=0}^{n-1} \int Q^i(1)\1_{U_n} dm\\
&=n \mu(U_n)+\sum_{i=0}^{n-1} \int Q^i(1)\1_{U_n} dm.
\end{align*}
The result follows if we show that the second term on the r.h.s. goes to 0, as $n\to\infty$. This follows easily since
$$
\sum_{i=0}^{n-1} \int Q^i(1)\1_{U_n} dm\leq \sum_{i=0}^{n-1} \alpha^i \int \1_{U_n} dm =\frac{1-\alpha^n}{1-\alpha}m(U_n)\xrightarrow[n\to\infty]{}0.
$$
\end{proof}

Since
$$
\sum_{i=0}^{n-1}m(X_i>u_n)=\sum_{i=0}^{n-1}\int_{U_n} \Pi_i(1)dm=\sum_{i=0}^{n-1}\int_{U_n} P^i(1)dm+ \sum_{i=0}^{n-1}\int_{U_n} \Pi_i(1)-P^i(1)dm,
$$
then condition \eqref{F_n} holds if we prove that the second term on the r.h.s. goes to 0 as $n\to\infty$.

Let $\varepsilon>0$ be arbitrary. Now, since $\xi>1$ then $\sum_{i\geq0} \frac{\log i}{i^\xi}<\infty$, so there exists $N\in\N$ such that $C_0\sum_{i\geq N} \frac{\log i}{i^\xi}<\eps/2$.

On the other hand, using the Lasota-Yorke inequalities for both $\Pi$ and $P$, we have that there exists some $C>0$ such that $| \Pi_i(1)-P^i(1)|\leq C$, for all $i\in\N$. Let $n$ be sufficiently large so that $C N m(U_n)<\eps/2$. Then
\begin{align*}
 \sum_{i=0}^{n-1}\int_{U_n} \Pi_i(1)-P^i(1)dm&= \sum_{i=0}^{N-1}\int_{U_n} \Pi_i(1)-P^i(1)dm+ \sum_{i=N}^{\infty}\int_{U_n} \Pi_i(1)-P^i(1)dm\\
 &\leq C N m(U_n)+ C_0\sum_{i\geq N} \frac{\log i}{i^\xi} <\eps/2+\eps/2=\eps.
\end{align*}

\subsection{Verification of $\D_q(u_n)$}
\label{subsec:D-check}
We start by proving the following statement about decay of correlations, which is just a slightly more general statement then the one proved in \cite[Section~3]{CR07}. \begin{proposition}
\label{prop:decay-correlations}
Let $\phi\in BV$ and $\psi\in L^1(m)$. Then for the $\beta$ transformations $T_n=T_{\beta_n}$ we have that
\[
\left|\int \phi\circ \TF_i\psi \circ \TF_{i+t}dm-\int \phi\circ \TF_i dm\int  \psi \circ \TF_{i+t} dm \right|\leq \mbox{B} \lambda ^t\|\phi\|_{BV}\|\psi\|_{1},
\]
for some $\lambda<1$ and $\mbox{B}>0$ independent of $\phi$ and $\psi$.
\end{proposition}

\begin{remark}
\label{rem:dec-corr}
Note that as it can be seen in \cite[Section~3]{CR07},
Proposition \ref{prop:decay-correlations}  holds for any sequence $T_{\beta_1}, T_{\beta_2}, \ldots$  of $\beta$ transformations and not necessarily only for the  ones that  satisfy condition \eqref{eq:beta-rate}.
\end{remark}

\begin{proof}
Using the adjoint property, write
\begin{align*}
D\!C(\phi,\psi,i,t):=&\int \phi\circ \TF_i\psi \circ \TF_{i+t}dm-\int \phi\circ \TF_i dm\int  \psi \circ \TF_{i+t} dm \\
=&\int \psi P_{i+t}\ldots P_{i+1}(\phi\Pi_{i}(1)) dm-\int \phi \Pi_i(1)dm\int  \psi \Pi_{i+t}(1) dm .
\end{align*}

Using the fact that the Perron-Frobenius operators preserve integrals we have
$$
\int \phi \Pi_i(1)dm\int  \psi \Pi_{i+t}(1) dm=\iint   \psi \Pi_{i+t}(1) dm \,P_{i+t}\ldots P_{i+1}(\phi\Pi_{i}(1)) dm.
$$
By linearity we also have
$$
\int \phi \Pi_i(1)dm\int  \psi \Pi_{i+t}(1) dm=\int  \psi P_{i+t}\ldots P_{i+1}\left(\int \phi \Pi_i(1)dm \,\Pi_{i}(1)\right) dm.
$$
Again linearity and preservation of the integrals allow us to write:
$$
\int \phi \Pi_i(1)dm\int  \psi \Pi_{i+t}(1) dm=\iint  \psi \Pi_{i+t}(1) dm \,P_{i+t}\ldots P_{i+1}\left(\int \phi \Pi_i(1)dm \,\Pi_{i}(1)\right)dm.
$$
Consequently we have
\begin{align*}
D\!C(\phi,\psi,i,t)&=  \int \psi P_{i+t}\ldots P_{i+1}(\phi\Pi_{i}(1)) dm- \iint   \psi \Pi_{i+t}(1) dm \,P_{i+t}\ldots P_{i+1}(\phi\Pi_{i}(1)) dm\\
&-\int  \psi P_{i+t}\ldots P_{i+1}\left(\int \phi \Pi_i(1)dm \,\Pi_{i}(1)\right) dm\\&+ \iint  \psi \Pi_{i+t}(1) dm \,P_{i+t}\ldots P_{i+1}\left(\int \phi \Pi_i(1)dm \,\Pi_{i}(1)\right)dm\\
&= \int\left(\psi-\int\psi \Pi_{i+t}(1)dm\right)\,P_{i+t}\ldots P_{i+1}\left(\Pi_i(1)\left(\phi-\int\phi \Pi_{i}(1)dm\right)\right).
\end{align*}
Let $\tilde\phi= \phi-\int\phi \Pi_{i}(1)dm$. Observe that $\int \Pi_i(1)\tilde \phi dm=0.$ This means that the observable function $\Pi_i(1)\tilde \phi\in \mathcal V_0$, where $\mathcal V_0$ is the set of functions with 0 integral that was defined in \cite[Lemma~2.12]{CR07}. Moreover, by
(DFLY), there exists a constant $C_0$ independent of $\phi$ and $\psi$ such that $\|\Pi_i(1)\tilde \phi\|_{BV}\leq 3C_0\|\phi\|_{BV}$.

As it has been shown in \cite[Section~3]{CR07}, condition (Dec) of the same paper is satisfied for any sequence of $\beta$ transformations as considered here. It follows that for all $g\in \mathcal V_0$ and $i\in\N$ we have that $\|P_{i+t}\ldots P_{i+1}(g)\|_{BV}\leq K\lambda ^t \|g\|_{BV}$, for some $K>0$ and $\lambda<1$ independent of $g$, which applied to $\Pi_i(1)\tilde \phi$ gives:
\begin{equation}
\label{eq:spectral-gap}
\|P_{i+t}\ldots P_{i+1}(\Pi_i(1)\tilde \phi)\|_{BV}\leq 3KC_0\lambda ^t \|\phi\|_{BV}.
\end{equation}
Let $\tilde \psi=\psi-\int\psi \Pi_{i+t}(1)dm$. Again, by \cite[(2.4)]{CR07}, we have $\|\tilde\psi\|_1\leq 2C_0 \|\psi\|_1$. Hence, using \eqref{eq:spectral-gap} we obtain
\begin{align*}
|D\!C(\phi,\psi,i,t)|&=\left|\int\tilde\psi\,P_{i+t}\ldots P_{i+1}\left(\Pi_i(1)\tilde\phi\right)dm\right|\\
	&\leq \|P_{i+t}\ldots P_{i+1}(\Pi_i(1)\tilde \phi)\|_{BV}\int |\tilde\psi| dm\\
	&\leq 6KC_0^2 \lambda ^t \|\phi\|_{BV} \|\psi\|_1.
\end{align*}
\end{proof}

Condition $\D_q(u_{n,i})$ follows from Proposition~\ref{prop:decay-correlations} by taking for each $i\in\N$,
\begin{equation*}
\phi_i=\I_{D_{n,i}^{(q)}} \mbox{ and }\psi_i= \I_{D_{n,i+t}^{(q)}}.\I_{D_{n,i+t+1}^{(q)}}\circ T_{i+t+1}.\cdots.\I_{D_{n,i+t+\ell}^{(q)}}\circ T_{i+t+\ell}\circ\ldots\circ T_{i+t+1},
\end{equation*}
where for every $j\in\N$ we define
\begin{equation}
\label{eq:D-def}
D_{n,j}^{(q)}= U_n\cap T_{j+1}^{-1}(U_n^c)\cap\ldots\cap T_{j+q}^{-1}(U_n^c).
\end{equation}
Since we assume that \eqref{eq:beta-rate} holds, there exists a constant $C>0$ depending on $q$ but not on $i$ such that $\|\phi_i\|_{BV}<C$. Moreover, it is clear that $\|\psi_i\|\leq 1$. Hence,
\begin{multline*}
\left|\p\left(\A_{n,i}\cap
 \mathscr W_{i+t,\ell}\left(\mathbb A_n^{(q)}\right) \right)-\p\left(\A_{n,i}\right)
  \p\left(\mathscr W_{i+t,\ell}\left(\mathbb A_n^{(q)}\right)\right)\right|=\\\left|\int \phi_i\circ \TF_i\psi_i \circ \TF_{i+t}dm-\int \phi_i\circ \TF_i dm\int  \psi_i \circ \TF_{i+t} dm \right|\leq \mbox{const}\; \lambda^t.
\end{multline*}
Thus, if we take $\gamma_i(q,n,t)=\mbox{const} \lambda^t$ and $t_n=(\log n)^2$ condition $\D_q(u_{n,i})$ is trivially satisfied.

\subsection{Verification of condition $\D'_q(u_n)$} We start by noting that we may neglect the first $n^\gamma$ random variables of the process $X_0, X_1, \ldots$, where $\gamma$ is such that $\gamma\xi>1$, for $\xi$ given as in \eqref{eq:beta-rate}.

In fact, by Lemma~\ref{lem:time-gap-1} and (DFLY)
we have
\begin{align*}
m(\max\{X_{n^\gamma}, \ldots, X_{n-1}\}\leq u_n)-m(M_n\leq u_n)&\leq \sum_{i=0}^{n^{\gamma}-1} m(X_i>u_n)= \sum_{i=0}^{n^{\gamma}-1} \int \I_{U_n}\Pi_i(1) dm\\
&\leq C_0 n^{\gamma}m(U_n)\xrightarrow[n\to\infty]{}0.
\end{align*}
This way, we simply disregard the $n^\gamma$ random variables of $X_0, X_1, \ldots$ and start the blocking procedure, described in Section~\ref{subsec:blocks}, in $X_{n^\gamma}$ by taking $\L_0=n^\gamma$. We split the remaining $n-n^{\gamma}$ random variables into $k_n$ blocks as described in Section~\ref{subsec:blocks}. Our goal is to show that
$$
S_n':=\sum_{i=1}^{k_n} \sum_{j=0}^{\ell_i-1} \sum_{r>j}^{\ell_i-1}m(\A_{\L_{i-1}+j}\cap \A_{\L_{i-1}+r})
$$
goes to $0$.

We define for some $i,n,q\in\N_0$,
\begin{align*}
&R_{n,i}^{(q)}:=\min\left\{j>i:\; \I_{A_i^{(q)}}\cdot\I_{A_j^{(q)}}(x)>0 \mbox{ for some $x\in[0,1]$}\right\},\\
&\tilde R_n^{(q)}:=\tilde R_n^{(q)}(n^\gamma)=\min\{R_{n,i}^{(q)},\; i=n^\gamma, \ldots,n\},\\
&L_n=\max\{\ell_{n,i},\; i=1,\ldots, k_n\}.
\end{align*}

We have
$$
S_n'\leq \sum_{i=n^\gamma}^{n}\sum_{j>i+R_{n,i}^{(q)}}^{L_n} m\left(A_i^{(q)}\cap A_j^{(q)}\right)= \sum_{i=n^\gamma}^{n}\sum_{j>i+R_{n,i}^{(q)}}^{L_n} \int \I_{D_{n,i}^{(q)}}\circ \TF_i \cdot \I_{D_{n,j}^{(q)}}\circ \TF_j\, dm,
$$
where $D_{n,i}^{(q)}$ and $D_{n,j}^{(q)}$ are given as in \eqref{eq:D-def}. Using Proposition~\ref{prop:decay-correlations}, with $\phi=\I_{D_{n,i}^{(q)}}$ and $\psi= \I_{D_{n,j}^{(q)}}$ and the adjoint property of the operators, it follows that
$$
\int \I_{D_{n,i}^{(q)}}\circ \TF_i \cdot \I_{D_{n,j}^{(q)}}\circ \TF_j\, dm\leq \int \I_{D_{n,i}^{(q)}} \Pi_i(1)dm  \int \I_{D_{n,j}^{(q)}} \Pi_j(1)dm+B \lambda^{j-i}\|\I_{D_{n,i}^{(q)}}\|_{BV}\|\I_{D_{n,j}^{(q)}}\|_1.
$$
Using (DFLY)
and since there exists some $C_2>0$ (independent of $n$) such that $\|\I_{D_{n,i}^{(q)}}\|_{BV}\leq C_2$, we have
$$
\int \I_{D_{n,i}^{(q)}}\circ \TF_i \cdot \I_{D_{n,j}^{(q)}}\circ \TF_j\, dm\leq C_0^2m(U_n)^2+BC_2\lambda^{j-1}m(U_n).
$$
Hence,
\begin{align*}
S_n'&\leq \sum_{i=n^\gamma}^{n}\sum_{j\geq i+R_{n,i}^{(q)}}^{L_n} \left(C_0^2m(U_n)^2+BC_2\lambda^{j-1}m(U_n)\right)\leq C_0^2nL_nm(U_n)^2+BC_2 m(U_n) n \sum_{k\geq \tilde R_n^{(q)}}^{L_n} \lambda^k\\
&\leq C_0^2nL_nm(U_n)^2+BC_2 m(U_n) n \lambda^{\tilde R_n^{(q)}}\frac1{1-\lambda}.
\end{align*}
Now we show that
\begin{equation}
\label{eq:Ln-estimate}
L_n=\frac{n}{k_n}(1+o(1)).
\end{equation}
To see this, observe that each $\ell_{n_i}$ is defined, in this case, by the largest integer $\ell_n$ such that
$
\sum_{j=s}^{s+\ell_n-1}m(X_j>U_n)\leq \frac{1}{k_n}\sum_{j=n^\gamma}^{n-1}m(X_j>u_n).
$
Using \eqref{eq:iterated-measure}, it follows that $\ell_n\mu(U_n)(1+o(1))\leq \frac{n-n^\gamma}{k_n}\mu(U_n)(1+o(1)).$ On the other hand, by definition of $\ell_n$ we must have $\sum_{j=s}^{s+\ell_n-1}m(X_j>U_n)>\frac{1}{k_n}\sum_{j=n^\gamma}^{n-1}m(X_j>u_n)-m(X_{s+\ell_n}>u_n).$ Using \eqref{eq:iterated-measure} again, we have $\ell_n\mu(U_n)(1+o(1))> \frac{n-n^\gamma}{k_n}\mu(U_n)(1+o(1))-\mu(U_n)(1+o(1))$. Together with the previous inequality, \eqref{eq:Ln-estimate} follows at once.

Using estimate \eqref{eq:Ln-estimate}, the fact that $\lim_{n\to\infty}n\mu(U_n)=\tau$ and $h\in BV$, we have that there exists some positive constant $C$ such that
$$
S_n'\leq C\left(\frac{1}{k_n}+\lambda^{\tilde R_n^{(q)}}\right).
$$

In order to prove that $\D'_q(u_n)$ holds, we need to show that $\tilde R_n^{(q)}\to\infty$, as $n\to\infty$, for all $q\in\N_0$.
To do that we have to split the proof in several cases. First, we have to consider the cases when the orbit of $\zeta$ hits $1$ or not. Then for each of the previous two cases, we have to consider if $\zeta$ is periodic or not.

We will consider that the maps $T_i$, for all $i\in\N_0$, are defined in $S^1$ by using the usual identification $0\sim1.$ Observe that the only point of discontinuity of such maps is $0\sim1$. Moreover, $\lim_{x\to 0^+}T_i(x)=0$ and $\lim_{x\to1^-}T_i(x)=\beta_i-\lfloor \beta_i\rfloor.$

\subsubsection{The orbit of $\zeta$ by the unperturbed $T_\beta$ map does not hit 1} We mean that for all $j\in\N_0$ we have $T^j(\zeta)\neq 1$.

\paragraph{The orbit of $\zeta$ is not periodic.} \label{subpar:aperiodic} In this case, for all $j\in\N$,  we have that $T^j(\zeta)\neq \zeta$, we  take $q=0$ and in particular $D_{n,i}^{(q)}=U_n$, for all $i\in\N_0$.
Let $J\in\N$.

We will check that for $n$ sufficiently large $\tilde R_n^{(q)}>J$. Since $\zeta$ is not periodic, there exists some $\epsilon>0$ such that $\min_{j=1,\ldots J}\dist(T^j(\zeta), \zeta)>\epsilon$. Let $N_1\in\N$ be sufficiently large so that for all $i\geq N_1$, we have
$$
\min_{j=1,\ldots J}\dist(T_{i+j}\circ\ldots\circ T_i(\zeta), T^j(\zeta))<\epsilon/4.
$$
Let $N_2\in\N$ be sufficiently large so that for all $i\geq N_2$ we have
$$
\mbox{diam}(T_{i+J}\circ\ldots\circ T_i(U_n))<\epsilon/4.
$$
This way for all $i\geq\max\{N_1,N_2\}$, for all $x\in U_n$ and for all $j\leq J$ we have
$$
\dist(T_{i+j}\circ\ldots\circ T_i(x), \zeta)>\epsilon/2.
$$
Hence, as long as $n^\gamma>\max\{N_1,N_2\}$ we have $\tilde R_n^{(q)}>J$.

Note that for this argument to work we only need that $\beta_n\to \beta$ and the stronger restriction imposed by \eqref{eq:beta-rate}  is not necessary.

\paragraph{The orbit of $\zeta$ is periodic.} \label{subpar:periodic} In this case, there exists $p\in\N$,  such that $T^j(\zeta)\neq \zeta$ for all $j<p$ and $T^p(\zeta)=\zeta$. We  take $q=p$.

Let \begin{equation}
\label{eq:def-eps-n}
\eps_n:=|\beta_{n^\gamma}-\beta|.\end{equation} By \eqref{eq:beta-rate} and choice of $\gamma$, we have that $\eps_n=o(n^{-1})$. Also let $\delta>0$, be such that $B_\delta(\zeta)$ is contained on a domain of injectivity of all $T_i$, with $i\geq n^\gamma$.

Let $J\in\N$ be chosen. Using a continuity argument, we can  show that there exists $C:=C(J,p)>0$ such that
$$
\dist(T_{i+j}\circ\ldots\circ T_{i+1} (\zeta), T^j(\zeta))< C\eps_n, \mbox{ for all $i=1, \ldots, J$}
$$
and moreover $U_n\cap T_{i+j}\circ\ldots\circ T_{i+1}(U_n)=\emptyset$, for all $j\leq J$ such that $j/p-\lfloor j/p\rfloor>0$.

We want to check that if $x\in A_i^{(q)}$ for some $i\geq n^\gamma$, \ie $\TF_i(x)\in D_{n,i}^{(q)}$, then $x\notin A_{i+j}^{(q)}$, for all $j=1, \ldots, J$, \ie $\TF_{i+j}(x)\notin D_{n,i+j}^{(q)}\subset U_n$, for all such $j$. By the assumptions above, we only need to check the latter for all $j=1, \ldots, J$ such that $j/p-\lfloor j/p\rfloor=0$, \ie for all $j=sp$, where $s=1, \ldots, \lfloor J/p\rfloor$.

By definition of $A_i^{(q)}$ the statement is clearly true when $s=1$. Let us consider now that $s>1$ and let $x\in A_i^{(q)}$. We may write
$$
\dist(\TF_{i+sp}(x), T_{i+sp}\circ\ldots\circ T_{i+p+1}(\zeta))>(\beta-\eps_n)^{(s-1)p}\dist(\TF_{i+p}(x),\zeta).
$$
On the other hand,
$$
\dist(T_{i+sp}\circ\ldots\circ T_{i+p+1}(\zeta), \zeta)\leq C\eps_n.
$$
Hence,
\begin{align*}
\dist(\TF_{i+sp}(x), \zeta)&\geq \dist(\TF_{i+sp}(x), T_{i+sp}\circ\ldots\circ T_{i+p+1}(\zeta))- \dist(T_{i+sp}\circ\ldots\circ T_{i+p+1}(\zeta), \zeta)\\
&\geq (\beta-\eps_n)^{(s-1)p}\dist(\TF_{i+p}(x),\zeta)-C\eps_n\\
&\geq (\beta-\eps_n)^{(s-1)p}\frac{m(U_n)}{2}-C\eps_n, \mbox{ since $x\in A_i^{(q)}\Rightarrow \TF_{i+p}(x)\notin U_n$}\\
&>\frac{m(U_n)}{2}, \mbox{ for $n$ sufficiently large, since $\eps_n=o(n^{-1})$.}
\end{align*}
This shows that $\TF_{sp+i}(x)\notin U_n$, which means that  $\TF_{sp+i}(x)\notin D_{n,i}^{(q)}$ and hence $x\notin A_{i+sp}^{(q)}$.

\subsubsection{$\zeta=0\sim1$} In this case we proceed in the same way as in \cite[Section~3.3]{AFV15}, which basically corresponds considering two versions of the same point: $\zeta^+=0$ and $\zeta^-=1$. Note that $\zeta^+$ is a fixed point for all maps considered and $\zeta^-$ may or not be periodic. So we split again into two cases.

\paragraph{1 is not periodic} This means that $T^i(1)\neq \zeta$ for all $i\in\N$. Note that $U_n$ can be divided into $U_n^+$ which corresponds to the bit having $0$ at its left border and $U_n^-$ which corresponds to the interval with $1$ as its endpoint. In this case, $q=1$ and $D_{n,i}^{(1)}$ has two connected components  one of them being $U_n^-$. Let $J\in\N$ be fixed as before. A continuity argument as the one used in Paragraph~\ref{subpar:aperiodic}, allows us to show that the points of $U_n^-$ do not return before $J$ iterates. An argument similar to the one used in Paragraph~\ref{subpar:periodic} would allow us to show also that the points of the other connected component of $D_{n,i}^{(1)}$ do not return to $U_n$ before time $J$, also.

\paragraph{1 is periodic} This means that there exists $p\in\N$ such that $T^i(1)\neq \zeta$ for all $i<p$ and $T^p(1)=\zeta$. In this case, we need to take $q=p$ and observe that $D_{n,i}^{(q)}$ has again two connected components, one to the right of 0 and the other to the left of 1, where none of the two points belongs to the set. The argument follows similarly as in the previous paragraph, except that this time both sides require mimicking the argument used in Paragraph~\ref{subpar:periodic}. Note that, the maps are orientation preserving so there is no switching as described in  \cite[Section~3.3]{AFV15}.

\subsection{Verification of condition \eqref{eq:EI}} We only need to verify \eqref{eq:EI}, when $\zeta$ has some sort of periodic behaviour. Let $\eps_n$ be defined as in \eqref{eq:def-eps-n}. Let $\delta_n$ be such that $U_n=B_{\delta_n}(\zeta)$. For simplicity, we assume that we are using the usual Riemannian metric so that we have a symmetry of the balls, which means that $|U_n|=m(U_n)=2\delta_n$.

Let us assume first that $\zeta$ is a periodic point of prime period $p$ with respect to the unperturbed map $T=T_\beta$ and the orbit of $\zeta$ does not hit $0\sim 1$. In this case, we take $q=p$, $\theta=1-\beta^{-p}$ and check \eqref{eq:EI}.

Using a continuity argument we can show that there exists $C:=C(J,p)>0$ such that
$$
\dist(T_{i+p}\circ\ldots\circ T_{i+1} (\zeta), \zeta)< C\eps_n.
$$
We define two points $\xi_u$ and $\xi_l$ of $B_{\delta_n}(\zeta)$ on the same side with respect to $\zeta$ such that $\dist(\xi_u,\zeta)=(\beta-\eps_n)^{-p}\delta_n+C\eps_n$ and $\dist(\xi_l,\zeta)=(\beta+\eps_n)^{-p}\delta_n-(\beta+\eps_n)^{-p}C\eps_n$. Recall that for all $i\geq n^\gamma$, we have that $(\beta-\eps_n)\leq \beta_i\cdot\ldots\cdot\beta_{i+p}\leq
(\beta+\eps_n)$.

Since we are composing $\beta$ transformations, then for all $i\geq n^\gamma$, we have $\dist(T_{i+p}\circ\ldots\circ T_i(\xi_u), T_{i+p}\circ\ldots\circ T_i(\zeta))\geq \delta_n+(\beta-\eps_n)^p C\eps_n$. Using the triangle inequality it follows that
$$
\dist(T_{i+p}\circ\ldots\circ T_i(\xi_u), \zeta)\geq \delta_n.
$$
Similarly, $\dist(T_{i+p}\circ\ldots\circ T_i(\xi_l), T_{i+p}\circ\ldots\circ T_i(\zeta))\leq \delta_n-C\eps_n$ and
$$
\dist(T_{i+p}\circ\ldots\circ T_i(\xi_l), \zeta)\leq \delta_n.
$$
If we assume that both $\xi_u$ and $\xi_l$ are on the right hand side with respect to $\zeta$ and $\xi_u^*$ and $\xi_l^*$ are the corresponding points on the left hand side of $\zeta$, then
$$
(\zeta-\delta_n,\xi_u^*]\cup[\xi_u, \zeta+\delta_n)\subset D_{n,i}^{(p)}\subset (\zeta-\delta_n,\xi_l^*]\cup[\xi_l, \zeta+\delta_n).
$$
Hence,
$$
\delta_n-(\beta-\eps_n)^{-p}\delta_n-C\eps_n\leq \frac12m(D_{n,i}^{(p)})\leq \delta_n- (\beta+\eps_n)^{-p}\delta_n+(\beta+\eps_n)^{-p}C\eps_n.
$$
Since $\eps_n=o(n^{-1})=o(\delta_n)$ then we easily get that
\begin{equation*}
\lim_{n\to\infty}\frac{m(D_{n,i}^{(p)})}{m(U_n)}=1-\beta^{-p}.
\end{equation*}

Now, observe that by \eqref{eq:iterated-measure}, $m(A_{n,i}^{(p)})=m(\TF_i^{-1}(D_{n,i}^{(p)}))=\mu(D_{n,i}^{(p)})+o(n^{-1})$ and $m(X_i>u_n)=\mu(U_n)+o(n^{-1})$. Hence, we have that
\begin{equation*}
\lim_{n\to\infty}\frac{m(A_{n,i}^{(p)})}{m(X_i>u_n)}=\lim_{n\to\infty}\frac{\mu(D_{n,i}^{(p)})}{\mu(U_n)}.
\end{equation*}

The density $\frac{d\mu}{dm}$, which can be found in \cite[Theorem~2]{P60}, is sufficiently regular so that, as in \cite[Section~7.3]{FFT15}, one can see that
$$
\lim_{n\to\infty}\frac{\mu(D_{n,i}^{(p)})}{\mu(U_n)}=\lim_{n\to\infty}\frac{m(D_{n,i}^{(p)})}{m(U_n)}.
$$
It follows that
$$
\lim_{n\to\infty}\frac{m(A_{n,i}^{(p)})}{m(X_i>u_n)}=1-\beta^{-p}.
$$
Since, as we have seen in \eqref{eq:Ln-estimate}, we can write that $\ell_{n,i}=\frac n{k_n}(1+o(1))$, then the previous equation can easily be used to prove  that condition \eqref{eq:EI} holds, with $\theta=1-\beta^{-p}$.

In the case $\zeta=0\sim1$, the argument follows similarly but this time we have to take into account the fact that the density is discontinuous at $0\sim1$. By  \cite{P60} we have that $$\frac{d\mu}{dm}(x)=\frac1{M(\beta)}\sum_{x<T^n (1)} \frac1{\beta^n},$$ where $M(\beta):=\int_0^1\sum_{x<T^n (1)} \frac1{\beta^n}dm$.  In this case, we have $\theta=\frac{d\mu}{dm}(0)(1-\beta^{-1})+\frac{d\mu}{dm}(1)$ if $1$ is not periodic and $\theta=\frac{d\mu}{dm}(0)(1-\beta^{-1})+\frac{d\mu}{dm}(1)(1-\beta^{-p})$ if $1$ is periodic of period $p$.

\subsection{An example with an EI equal to 1 at periodic points}
\label{subsec:counter-example}

In the previous subsections, we used \eqref{eq:beta-rate}, which imposes a fast accumulation rate of $\beta_n$ to $\beta$, to show that the EI equals the EI observed for the unperturbed dynamics. If this condition fails then the EI for the sequential dynamics does not need to be the same as the one of the original system.

Let $\beta=5/2$ and $T=T_\beta=5/2x \mbox{ mod } 1$. Let $\zeta=2/3$. Note that $T(2/3)=2/3$. Consider a sequence $\beta_j=5/2+\eps_j$, with $\eps_j=j^{-\alpha}$, where $\alpha<1$. Note that $1/n=o(\eps_n)$.

Observe that $T_j(2/3)=2/3+O(\eps_j)$. Also note that, since we are choosing, deliberately, $\eps_j>0$ for all $j$, then the orbit of $\zeta$ is being pulled to the right everytime we iterate. Moreover, by letting $j$ be sufficiently large we can keep it inside a small neighbourhood of $2/3$ at least up to a certain fixed number of iterates $J\in\N$.

For $\delta>0$, we have that $T_j(2/3-\delta)=2/3+O(\delta)+O(\eps_j)$. So if we take $\delta=\delta_n$ such that $B_{\delta_n}(\zeta)=U_n$ then $\delta_n=O(1/n)$ and we see that if $j$ and $n$ are sufficiently large then $T_j(2/3-\delta_n)>2/3+\delta_n$. Hence, by continuity, for some fixed $J\in\N$, we can show that for $j$ and $n$ sufficiently large then for all $i=1,\ldots, J$ we have $T_{j+i}\circ\ldots\circ T_{j}(U_n)\cap U_n=\emptyset$. This means that we would be able to show that $\D'_0(u_n)$ would hold with $A_{n,i}^{(q)}=U_n$ (meaning that $q=0$).

The conclusion then is that at $\zeta=2/3$, although for the unperturbed system $T$ shows an EI equal to $1-2/5=3/5$, for the sequential systems chosen as above the EI is equal to $1$.

\begin{remark}
Note that condition  \eqref{eq:beta-rate} was used to prove \eqref{F_n} so, in this case, we may need to use different $u_{n,i}$ for each $i$ but, since the invariant measure of each  $T_i$ is equivalent to Lebesgue measure, the corresponding $\delta_{n,i}$ still satisfies $\delta_{n,i}=O(1/n)$ for all $i\in\N$.
\end{remark}

\section{Random fibered dynamical systems}
We now provide a second example of non-stationary dynamical systems, this time arising from  suitable random perturbations.\\ We consider a probability space $(\Omega, \mathcal{G}, P)$ with an invertible $P$-preserving transformation $\vartheta:\Omega\rightarrow \Omega$; then we let $(\XI, \mathcal{F})$ another measurable space and $\Xi$ a measurable (with respect to the product $\mathcal{G}\times \mathcal{F}$) subset of $\XI \times \Omega$ with the fibers $\Xi^{\o}=\{\xi\in \XI: (\xi, \o)\in \Xi\}\in \mathcal{F}$. We define the (skew) map $s: \Xi\rightarrow \Xi$ by $s(\xi, \o)=(f_{\o}\xi, \vartheta\o)$, with $f_{\o}: \Xi^{\o}\rightarrow \Xi^{\vartheta\o}$ being measurable fiber maps with the composition rule
$$
f^n_{\o}: \Xi^{\o}\rightarrow \Xi^{\vartheta^n\o}, \
f^n_{\o}=f_{\vartheta^{n-1}\o}\circ\dots\circ f_{\o}.
$$
We also put
$$
f^j_{\vartheta^l\o}: \Xi^{\vartheta^{l}\o}\rightarrow \Xi^{\vartheta^{l+j}\o}; \ f^j_{\vartheta^l\o}=f_{\vartheta^{l+j-1}\o}\circ\dots\circ f_{\vartheta^l\o}.
$$
Moreover we set
$$
f^{-1}_{\vartheta^j\o}: \Xi^{\vartheta^{j+1}\o}\rightarrow \Xi^{\vartheta^{j}\o} \ \mbox{and} \
(f^k_{\o})^{-1}:=f^{-1}_{\o}\circ \dots \circ f^{-1}_{\vartheta^{k-1}\o}.
$$
This will allow us to introduce the $\sigma$-algebras
$
\mathcal{T}^{\o}_k:=(f^k_{\o})^{-1}\mathcal{T}^{\vartheta^k\o}_0
$
where $\mathcal{T}^{\vartheta^k\o}_0$ is the restriction of the $\sigma$-algebra $\mathcal{F}$ to $\Xi^{\o}\subset \XI.$\\

It is well known that a measure $\mu$ disintegrated with respect to the measure $P$ will be  $s$-invariant if the  conditional measures $\mu^{\o}$ will verify the quasi-invariant relation
\begin{equation}\label{qi}
(f_{\o})_*\mu^{\o}=\mu^{\vartheta\o}.
\end{equation}
An interesting case is whenever all the fibers $\Xi^{\o}$ coincide with the metric space $X$. In this case we can also define a marginal  measure $\mu$ on $X$ in the following way: given $A\subset X$, define
$$
\mu(A) = \tilde{\mu}(\Omega\times A) = \int_{\Omega} \mu^\omega(A)\,dP(\omega).
$$
Also in this case, the stochastic process is defined by
\begin{equation}
\label{eq:SP-random}
X_i=\varphi\circ f_\omega^i,
\end{equation}
 where $\varphi:X\to\R\cup\{+\infty\}$ is as in \eqref{eq:observable-form}. This stochastic process $X_0, X_1,\ldots$ is not necessarily stationary and, by \eqref{qi}, the distribution function of $X_i$ is given by
$$
F_i(u)=\mu^{\vartheta^i\o}(\{x\in X:\; \varphi(x)\leq u\}).
$$
In this setting, we will consider that the boundary levels $u_{n,0}, u_{n,1}, \ldots$ are such that $u_n=u_{n,0}=u_{n,1}=\ldots$, where $u_n$ is determined by the marginal measure $\mu$ so that
\begin{equation*}
\label{eq:un-fibered}
u_n=\inf\left\{u\in\R:\;\mu(\{x\in X:\; \varphi(x)\leq u\})\geq 1-\frac\tau n\right\}.
\end{equation*}

Then as a result of the theory developed in Section~\ref{sec:EVL-nonstationary}, we can write a quenched distributional limit for the partial maxima of the process $X_0, X_1,\ldots$. Namely, as a consequence of Theorem~\ref{thm:error-terms-general-SP-no-clustering} we have
\begin{corollary}
\label{cor:EVL-random}
Let $X_0, X_1, \ldots$ be a stationary stochastic process defined as above, based on the action of the fiber maps $f_\omega^n$. Assume that for $P$-a.e. $\omega\in\Omega$ conditions \eqref{Fmax} and \eqref{F_n} hold for some $\tau>0$. Assume that there exists $q\in\N_0$, defined as in \eqref{def:q}, and  \eqref{eq:EI} holds for $P$-a.e. $\omega\in\Omega$. Assume moreover that conditions $\D_q(u_{n,i})$ e $\D'_q(u_{n,i})$ are satisfied for $P$-a.e. $\omega\in\Omega$. Then
\begin{align*}
\lim_{n\to \infty}\mathbf \mu^\omega(\max\{X_0, \ldots, X_{n-1}\}\leq u_n)=e^{-\theta \tau}, \quad\text{for $P$-a.e. $\omega\in\Omega$}.
\end{align*}
\end{corollary}

To illustrate an application of the theory developed here and in particular of Corollary~\ref{cor:EVL-random}, we look into random subshifts.

\subsection{Random subshifts}
\label{subsec:random-subshift}

We consider the random subshifts studied in \cite{RSV14} and \cite{RT15}, in the setting of Hitting Times. Here we will keep using an Extreme Values approach and the statements can be seen as a translation of the corresponding results in \cite{RSV14,RT15}, in light of the connection between HTS and EVL proved in \cite{FFT10, FFT11}.

Since the target sets, in this example, are dynamically defined cylinders, we need to produce some adjustments to the definition of the observable and to the time scale, as in \cite[Section~5]{FFT11} (where the notion of cylinder EVL was introduced), in order to properly use an EVL approach. We return to this issue below. Meanwhile, we introduce the notions using mostly the notation of \cite{RT15}.

Let $(\Omega,\vartheta,P)$ be an invertible ergodic measure preserving system, set $X=\N^{\N_0}$ and let $\sigma:X\rightarrow X$ denote the shift. Let $A=\left\{A(\omega)=(a_{ij}(\omega)):\omega\in\Omega\right\}$ be a random transition matrix, \ie for any $\omega\in\Omega$, $A(\omega)$ is in an $\N\times \N$-matrix with entries in $\{0,1\}$, with at least one non-zero entry in each row and each column and such that $\omega\rightarrow a_{ij}(\omega)$ is measurable for any $i\in \N$ and $j\in \N$. For any $\omega\in\Omega$ define
\[
X_{\omega}=\{x=(x_0,x_1,\ldots):x_i\in\N \mbox{ and } a_{x_{i}x_{i+1}}(\vartheta^i\omega)=1 \mbox{ for all } i\in\N\}
\]
and
\[
{\mathcal E}=\{(\omega,x):\omega\in\Omega, x\in X_{\omega}\}\subset\Omega\times X.
\]
We consider the random dynamical system coded by the skew-product $S:{\mathcal E}\rightarrow{\mathcal E}$ given by $S(\omega,x)=(\vartheta\omega,\sigma x)$. While we allow infinite alphabets here, we nevertheless call $S$ a {\it random subshift of finite type} (SFT). Assume that $\nu$ is an $S$-invariant probability measure with marginal $P$ on $\Omega$. Then we let $(\mu^{\omega})_{\omega}$ denote its decomposition on $X_{\omega}$, that is, $d\nu(\omega,x)=d\mu_{\omega}(x)dP(\omega)$. The measures $\mu^{\omega}$ are called the {\it sample measures}. Note $\mu^{\omega}(A)=0$ if $A \cap X_{\omega}=\emptyset$. As before, we denote by $\mu=\int \mu^{\omega}dP$ the marginal of $\nu$ on $X$.

For any $y\in X$ we denote by $C_n(y)=\{z\in X:y_i=z_i \mbox{ for all } 0\leq i\leq n-1\}$ the $n$-cylinder that contains  $y$. Let $\F^n_0$ be the $\sigma$-algebra in $X$, generated by all the $n$-cylinders.

We assume the following: there are constants $h_0>0$, $c_0>0$ and a summable function $\psi$ such that for all $m$, $n$, $\kappa\in \N$, $A\in\F^n_0$ and $B\in\F^m_0$:

\begin{enumerate}
\item the marginal measure $\mu$ satisfies
\[
\left|\mu (A\cap \sigma^{-\kappa-n}B)-\mu(A)\mu(B)\right|\leq \psi(\kappa);
\]
\label{1}

\item for $P$-almost every $\omega\in\Omega$, if $y\in X_{\omega}$ and $n\geq 1$ then $c_0^{-1}\e^{-h_0n}\leq\mu(c_n(y))$;
\label{2}

\item for $P$-almost every $\omega\in\Omega$,
\[
\left|\mu^{\omega}(A\cap \sigma^{-\kappa-n}B)-\mu^{\omega}(A)\mu^{\vartheta^{n+\kappa}\omega}(B)\right|\leq\psi(\kappa)\mu^{\omega}(A)\mu^{\vartheta^{n+\kappa}\omega}(B);
\]
\label{3}

\item the sample measure satisfies
\[
\stackrel[\omega\in\Omega]{}{\esup}\sup_{x\in X}\mu^{\omega}(C_1(x))<1.
\]
\label{4}

\end{enumerate}

The following lemma has been proved in \cite{RT15}.

\begin{lemma}
For a random SFT such that assumptions \eqref{3} and \eqref{4} hold, there exist $c_1$, $c_2>0$ and $h_1>0$ such that for any $y\in X$, $n\geq 1$ and $m\geq 1$, for almost $P$-almost every $\omega\in\Omega$,
\[
\mu^{\omega}(C_n(y))\leq c_1 \e^{-h_1 n}
\]
and
\[
\sum_{k=m}^{n}\mu^{\omega}(C_n(y)\cap\sigma^{-k}C_n(y))\leq c_2 \e^{-h_1 m}\mu^{\omega}(C_n(y)).
\]
\end{lemma}

Since the target sets are cylinders, in order to state the result using an EVL approach, as mentioned earlier, we need to make some adjustments to the definition of the observable function and to the time scale. Hence, proceeding as in \cite[Section~5]{FFT11}, the stochastic process is defined by $X_i=\varphi\circ\sigma^i$, where $\varphi:X\to\R\cup\{+\infty\}$ instead of being given by \eqref{eq:observable-form} is given by
$$
\varphi(x)=g(\mu(C_{n(x)}(\zeta)),
$$
where $n(x):=\max\{j\in\N:\; x\in C_j(\zeta)\}$ and $g$ is as in Section~\ref{subset:observable-choice}. As in \cite[(5.5)]{FFT11} we let the sequence $(u_n)_{n\in\N}$ be such that $\{x\in X:\,\varphi(x)>u_n\}=C_n(\zeta)$. Moreover, for the time scale we use the sequence $(w_n)_{n\in\N}$ given by \cite[(5.6)]{FFT11}:
$$
w_n=[\tau \mu(\{x\in X:\,\varphi(x)>u_n\})],
$$
for some $\tau\geq 0$.

Now, we can apply Corollary~\ref{cor:EVL-random} to obtain the following result, which is a translation to the EVL setting of \cite[Theorem~2.2]{RT15}.
\begin{theorem}
\label{thm:random-subshift}
Assume \eqref{1}-\eqref{4} hold and there exists a constant $q>2\frac{h_0}{h_1}$ such that $\psi$ satisfies $\psi(\kappa)\kappa^q\rightarrow$ as $\kappa\rightarrow+\infty$. Let $\zeta\in X$. Then for $P$-almost every $\omega$, either

(a) $\zeta$ is a periodic point of period $p$ and if the limit $\theta:=\lim_{n\rightarrow\infty} \frac{\mu(C_n(\zeta)\setminus C_{n+p}(\zeta))}{\mu(C_n(\zeta))}$ exists, then for all $\tau\geq 0$ we have
\[
\lim_{n\rightarrow\infty} \mu^{\omega}\left(M_{w_n}\leq u_n\right)=\e^{-\theta \tau};
\]
or

(b) for all $\tau\geq 0$ we have
\[
\lim_{n\rightarrow\infty} \mu^{\omega}\left(M_{w_n}\leq u_n\right)=\e^{-\tau}.
\]
\end{theorem}

In order to use Corollary~\ref{cor:EVL-random} to prove Theorem~\ref{thm:random-subshift}, one needs to check that conditions \eqref{F_n}, $\D_q(u_{n,i})$, $\D_q'(u_{n,i})$ and \eqref{eq:EI}  hold for $P$-a.e. $\omega\in\Omega$.

Note that because of the adjustments required to the cylinder setting, for condition \eqref{F_n}, one needs to check that for $P$-a.e. $\omega\in\Omega$ we have
$$
\lim_{n\to\infty}\sum_{i=0}^{w_n} \mu^{\vartheta^i(\omega)}(C_n(\zeta))=\tau,
$$
which follows immediately from \cite[Lemma~4.5]{RT15}. In the same way, conditions $\D_q(u_{n,i})$, $\D_q'(u_{n,i})$ follow from \cite[Lemma~4.8]{RT15} and \cite[Lemma~4.9]{RT15} respectively and condition  \eqref{eq:EI} from the discussion in  \cite[Section~5]{RT15}.

\section{Concluding remarks}

%Both papers \cite{CR07, HNTV} dealt however with uniformly expanding maps, for which the transfer operators admits a spectral gap and the correlations decays exponentially. In a different direction, a class of sequential systems given by composition of non-uniformly expanding maps of  Pomeau-Manneville type was studied in
%\cite{AHNTV15},  by perturbing the slope at the indifferent fixed point $0.$ Polynomial decay of correlations was proved for particular classes of centered observables, which could also be interpreted as the decay of the iterates of the transfer operator on functions of zero (Lebesgue) average, and this fact is better known as {\em loss of memory.} In the successor paper \cite{NTV15},
%a (non-stationary) central limit theorem was  shown for sums of centered observables and with respect to the Lebesgue measure. We continue here the statistical analysis of these indifferent transformations by proving   the existence of extreme value distributions under suitable normalization for the threshold of the exceedances.\\

The sequential systems considered in this paper were built on uniformly expanding maps, for which the
transfer operators admits a spectral gap and the correlations decay exponentially. In a
different direction, a class of sequential systems given by composition of non-uniformly
expanding maps of Pomeau-Manneville type was studied in \cite{AHNTV15}, by perturbing the
slope at the indifferent fixed point $0$. Polynomial decay of correlations was proved for particular
classes of centred observables, which could also be interpreted as the decay of the
iterates of the transfer operator on functions of zero (Lebesgue) average, and this fact
is better known as \emph{loss of memory}. In the successor paper \cite{NTV15}, a (non-stationary)
central limit theorem was shown for sums of centred observables and with respect to the
Lebesgue measure. In the forthcoming paper \cite{FFV2016} we will continue  the statistical analysis of these indifferent transformations
by proving the existence of extreme value distributions under suitable normalisation
for the threshold of the exceedances.

\section*{Acknowledgements}
ACMF was partially supported by FCT (Portugal) grant SFRH/BPD/66174/2009. JMF was partially supported by FCT grant SFRH/BPD/66040/2009. All these grants are financially supported by the program POPH/FSE.  ACMF and JMF were partially supported by FCT project FAPESP/19805/2014 and by CMUP (UID/MAT/00144/2013), which is funded by FCT with national (MEC) and European structural funds through the programs FEDER, under the partnership agreement PT2020. SV was supported by the ANR- Project {\em Perturbations} and by the project {\em Atracci\'on de Capital Humano Avanzado del Extranjero} MEC 80130047, CONICYT, at the CIMFAV, University of Valparaiso. All authors were partially  supported by FCT projects PTDC/MAT/120346/2010 and PTDC/ MAT-CAL/3884/2014, which are funded by national and European structural funds through the programs  FEDER and COMPETE. SV is grateful to  N. Haydn, M. Nicol and A. T\"or\"ok for several and fruitful discussions on sequential systems, especially in the framework of indifferent maps. JMF is grateful to M. Todd for fruitful discussions and careful reading of a preliminary version of this paper.  All authors acknowledge the Isaac Newton Institute for Mathematical Sciences, where this work was initiated during the program Mathematics for the Fluid Earth.

\bibliographystyle{abbrv}

\bibliography{Nonstationary.bib}

\begin{thebibliography}{10}

\bibitem{AHNTV15}
R.~Aimino, H.~Hu, M.~Nicol, A.~T{\"o}r{\"o}k, and S.~Vaienti.
\newblock Polynomial loss of memory for maps of the interval with a neutral
  fixed point.
\newblock {\em Discrete Contin. Dyn. Syst.}, 35(3):793--806, 2015.

\bibitem{AFLV11}
J.~F. Alves, J.~M. Freitas, S.~Luzzatto, and S.~Vaienti.
\newblock From rates of mixing to recurrence times via large deviations.
\newblock {\em Adv. Math.}, 228(2):1203--1236, 2011.

\bibitem{A98}
L.~Arnold.
\newblock {\em Random dynamical systems}.
\newblock Springer Monographs in Mathematics. Springer-Verlag, Berlin, 1998.

\bibitem{AFV15}
H.~Ayta{\c c}, J.~M. Freitas, and S.~Vaienti.
\newblock Laws of rare events for deterministic and random dynamical systems.
\newblock {\em Trans. Amer. Math. Soc.}, 367(11):8229--8278, 2015.

\bibitem{BV13}
W.~Bahsoun and S.~Vaienti.
\newblock Escape rates formulae and metastability for randomly perturbed maps.
\newblock {\em Nonlinearity}, 26(5):1415--1438, 2013.

\bibitem{BB84}
D.~Berend and V.~Bergelson.
\newblock Ergodic and mixing sequences of transformations.
\newblock {\em Ergodic Theory Dynam. Systems}, 4(3):353--366, 1984.

\bibitem{CHM91}
M.~R. Chernick, T.~Hsing, and W.~P. McCormick.
\newblock Calculating the extremal index for a class of stationary sequences.
\newblock {\em Adv. in Appl. Probab.}, 23(4):835--850, 1991.

\bibitem{C01}
P.~Collet.
\newblock Statistics of closest return for some non-uniformly hyperbolic
  systems.
\newblock {\em Ergodic Theory Dynam. Systems}, 21(2):401--420, 2001.

\bibitem{CR07}
J.-P. Conze and A.~Raugi.
\newblock Limit theorems for sequential expanding dynamical systems on
  {$[0,1]$}.
\newblock In {\em Ergodic theory and related fields}, volume 430 of {\em
  Contemp. Math.}, pages 89--121. Amer. Math. Soc., Providence, RI, 2007.

\bibitem{FHR11}
M.~Falk, J.~H{\"u}sler, and R.-D. Reiss.
\newblock {\em Laws of small numbers: extremes and rare events}.
\newblock Birkh\"auser/Springer Basel AG, Basel, extended edition, 2011.

\bibitem{FF08}
A.~C.~M. Freitas and J.~M. Freitas.
\newblock On the link between dependence and independence in extreme value
  theory for dynamical systems.
\newblock {\em Statist. Probab. Lett.}, 78(9):1088--1093, 2008.

\bibitem{FFT10}
A.~C.~M. Freitas, J.~M. Freitas, and M.~Todd.
\newblock Hitting time statistics and extreme value theory.
\newblock {\em Probab. Theory Related Fields}, 147(3-4):675--710, 2010.

\bibitem{FFT11}
A.~C.~M. Freitas, J.~M. Freitas, and M.~Todd.
\newblock Extreme value laws in dynamical systems for non-smooth observations.
\newblock {\em J. Stat. Phys.}, 142(1):108--126, 2011.

\bibitem{FFT12}
A.~C.~M. Freitas, J.~M. Freitas, and M.~Todd.
\newblock The extremal index, hitting time statistics and periodicity.
\newblock {\em Adv. Math.}, 231(5):2626--2665, 2012.

\bibitem{FFT15}
A.~C.~M. Freitas, J.~M. Freitas, and M.~Todd.
\newblock Speed of convergence for laws of rare events and escape rates.
\newblock {\em Stochastic Process. Appl.}, 125(4):1653--1687, 2015.

\bibitem{FFV2016}
A.~C.~M. Freitas, J.~M. Freitas, and S.~Vaienti.
\newblock Extreme value laws for non stationary, sequential intermittent
  systems.
\newblock In preparation.

\bibitem{GHW11}
C.~Gonz{\'a}lez-Tokman, B.~R. Hunt, and P.~Wright.
\newblock Approximating invariant densities of metastable systems.
\newblock {\em Ergodic Theory Dynam. Systems}, 31(5):1345--1361, 2011.

\bibitem{HNTV}
N.~Haydn, M.~Nicol, A.~T\"or\"ok, and S.~Vaienti.
\newblock Almost sure invariance principle for sequential and non-stationary
  dynamical systems.
\newblock Preprint arXiv:1406.4266.

\bibitem{HNVZ13}
N.~Haydn, M.~Nicol, S.~Vaienti, and L.~Zhang.
\newblock Central limit theorems for the shrinking target problem.
\newblock {\em J. Stat. Phys.}, 153(5):864--887, 2013.

\bibitem{HV09a}
H.~Hu and S.~Vaienti.
\newblock Absolutely continuous invariant measures for non-uniformly expanding
  maps.
\newblock {\em Ergodic Theory Dynam. Systems}, 29(4):1185--1215, 2009.

\bibitem{H83}
J.~H{\"u}sler.
\newblock Asymptotic approximation of crossing probabilities of random
  sequences.
\newblock {\em Z. Wahrsch. Verw. Gebiete}, 63(2):257--270, 1983.

\bibitem{H86}
J.~H{\"u}sler.
\newblock Extreme values of nonstationary random sequences.
\newblock {\em J. Appl. Probab.}, 23(4):937--950, 1986.

\bibitem{K86}
Y.~Kifer.
\newblock {\em Ergodic theory of random transformations}, volume~10 of {\em
  Progress in Probability and Statistics}.
\newblock Birkh\"auser Boston, Inc., Boston, MA, 1986.

\bibitem{K88}
Y.~Kifer.
\newblock {\em Random perturbations of dynamical systems}, volume~16 of {\em
  Progress in Probability and Statistics}.
\newblock Birkh\"auser Boston, Inc., Boston, MA, 1988.

\bibitem{K98}
Y.~Kifer.
\newblock Limit theorems for random transformations and processes in random
  environments.
\newblock {\em Trans. Amer. Math. Soc.}, 350(4):1481--1518, 1998.

\bibitem{KR14}
Y.~Kifer and A.~Rapaport.
\newblock Poisson and compound {P}oisson approximations in conventional and
  nonconventional setups.
\newblock {\em Probab. Theory Related Fields}, 160(3-4):797--831, 2014.

\bibitem{NTV15}
M.~Nicol, A.~T{\"o}r{\"o}k, and S.~Vaienti.
\newblock Central limit theorems for sequential and random intermittent
  dynamical systems.
\newblock Preprint arXiv:1510.03214, 2015.

\bibitem{N97}
X.-F. Niu.
\newblock Extreme value theory for a class of nonstationary time series with
  applications.
\newblock {\em Ann. Appl. Probab.}, 7(2):508--522, 1997.

\bibitem{P60}
W.~Parry.
\newblock On the {$\beta $}-expansions of real numbers.
\newblock {\em Acta Math. Acad. Sci. Hungar.}, 11:401--416, 1960.

\bibitem{R14}
J.~Rousseau.
\newblock Hitting time statistics for observations of dynamical systems.
\newblock {\em Nonlinearity}, 27(9):2377--2392, 2014.

\bibitem{RSV14}
J.~Rousseau, B.~Saussol, and P.~Varandas.
\newblock Exponential law for random subshifts of finite type.
\newblock {\em Stochastic Process. Appl.}, 124(10):3260--3276, 2014.

\bibitem{RT15}
J.~Rousseau and M.~Todd.
\newblock Hitting times and periodicity in random dynamics.
\newblock {\em J. Stat. Phys.}, 161(1):131--150, 2015.

\bibitem{S00}
B.~Saussol.
\newblock Absolutely continuous invariant measures for multidimensional
  expanding maps.
\newblock {\em Israel J. Math.}, 116:223--248, 2000.

\end{thebibliography}

\end{document}